\pdfoutput=1
\documentclass[reqno]{amsart}
\usepackage{hyperref, bbm, amssymb, mathtools, url, enumerate}
\usepackage{amsmath}
\usepackage{tikz}
\usepackage{tikz-cd}
\usepackage[scr]{rsfso}
\usepackage{bm}
\usepackage{amssymb}
\usepackage{bbm}
\usetikzlibrary{matrix,arrows,decorations.pathmorphing, decorations.markings, cd}
\usepackage[capitalise]{cleveref}
\usepackage{rotate}
\usepackage[alphabetic]{amsrefs}
\usepackage{microtype}

\DeclareMathAlphabet{\mymathbb}{U}{BOONDOX-ds}{m}{n}
\def\twocell[#1]{\arrow[#1, dash, phantom, "\Rightarrow"{scale=1.125, yshift=-.4pt, description, allow upside down, sloped, inner sep=0pt}]}

\tikzset{curve/.style={settings={#1},to path={(\tikztostart)
			.. controls ($(\tikztostart)!\pv{pos}!(\tikztotarget)!\pv{height}!270:(\tikztotarget)$)
			and ($(\tikztostart)!1-\pv{pos}!(\tikztotarget)!\pv{height}!270:(\tikztotarget)$)
			.. (\tikztotarget)\tikztonodes}},
	settings/.code={\tikzset{quiver/.cd,#1}
		\def\pv##1{\pgfkeysvalueof{/tikz/quiver/##1}}},
	quiver/.cd,pos/.initial=0.35,height/.initial=0}

\pgfdeclarelayer{bg}
\pgfsetlayers{bg,main}
\usetikzlibrary{calc}

\newtheorem{theorem}{Theorem}[section]

\newtheorem{corollary}[theorem]{Corollary}
\newtheorem{lemma}[theorem]{Lemma}
\newtheorem{proposition}[theorem]{Proposition}

\newtheorem{introthm}{Theorem}

\theoremstyle{definition}
\newtheorem*{claim*}{Claim}
\newtheorem{construction}[theorem]{Construction}
\newtheorem{variant}[theorem]{Variant}

\newtheorem{definition}[theorem]{Definition}

\newtheorem{example}[theorem]{Example}

\newtheorem{remark}[theorem]{Remark}
\newtheorem*{remark*}{Remark}

\newcommand{\qednow}{\pushQED{\qed}\qedhere\popQED}

\newcommand{\Cc}{{\mathcal{C}}}
\newcommand{\Dd}{{\mathcal{D}}}
\newcommand{\Ee}{{\mathcal{E}}}

\newcommand{\Mm}{{\mathcal{M}}}

\newcommand{\Ii}{{\mathcal{I}}}

\newcommand{\Ww}{{\mathcal{W}}}

\newcommand{\gl}{\textup{gl}}

\renewcommand{\phi}{\varphi}
\renewcommand{\epsilon}{\varepsilon}

\renewcommand{\S}{{{\mathscr S}}}
\DeclareMathOperator{\Sp}{Sp}
\DeclareMathOperator{\Spc}{Spc}
\DeclareMathOperator{\Cat}{Cat}
\DeclareMathOperator{\CAT}{CAT}

\DeclareMathOperator{\Fun}{Fun}

\DeclareMathOperator{\Ar}{Ar}

\DeclareMathOperator{\CMon}{CMon}

\DeclareMathOperator{\AdTrip}{AdTrip}
\DeclareMathOperator{\Span}{Span}

\DeclareMathOperator{\Mack}{Mack}

\newcommand{\cat}[1]{{\textup{#1}}}
\newcommand{\cato}[1]{{\textup{\textbf{#1}}}}

\newcommand{\op}{{\textup{op}}}

\DeclareMathOperator{\core}{\iota}

\newcommand{\id}{\textup{id}}

\newcommand{\fgt}{\textup{fgt}}

\newcommand{\Fin}{\hbox{$\mathcal F\kern-1.7pt\textit{in}$}}
\newcommand{\SymMonCat}{\cat{SymMonCat}}

\newcommand{\Glo}{\textup{Glo}}
\newcommand{\Fglo}{{\mathbb G}}
\newcommand{\Forb}{{\mathbb O}}

\newcommand{\ev}{{\textup{ev}}}

\newcommand{\ulhelper}[2]{\underline{\setbox0=\hbox{$#1#2$}\dp0=0.4pt \box0\relax}}
\newcommand{\ul}[1]{{\mathpalette\ulhelper{#1}}\hbox{\rule[-2pt]{0pt}{0pt}}}

\newcommand{\Swan}{{\textup{\textsc{Swan}}}}
\newcommand{\GammaS}{\Gamma{\mathscr S}}

\newcommand{\mySp}{{\mathscr S\!p}}

\newcommand{\iso}{\xrightarrow{\;\smash{\raisebox{-0.25ex}{\ensuremath{\scriptstyle\sim}}}\;}}

\tikzcdset{pullback/.style = {phantom, "\lrcorner"{very near start}}}

\newcommand\noloc{%
	\nobreak
	\mspace{6mu plus 1mu}
	{:}
	\nonscript\mkern-\thinmuskip
	\mathpunct{}
	\mspace{2mu}
}

\title[Mackey functors and classical equivariant $K$-theory]{Mackey functors\\ and classical equivariant $\bm K$-theory}
\author{Tobias Lenz}
\address{Mathematisches Institut, Rheinische Friedrich-Wilhelms-Universität Bonn, Endenicher Allee 60, 53115 Bonn, Germany}

\begin{document}
\begin{abstract}
    We show that the spectral Mackey functors associated to the equivariant algebraic $K$-theory spectra of Guillou--May and Merling (originally constructed using pointset models) can be described purely $\infty$-categorically in terms of the monoidal Borel construction of Barwick--Glasman--Shah and Hilman. We moreover show how Pützstück's global version of the Borel construction provides an analogous description of the global spectral Mackey functors arising from Schwede's global algebraic $K$-theory spectra. 

	Our arguments crucially rely on techniques from parametrized higher category theory as well as on structural results on global and equivariant $K$-theory to avoid any explicit computations. 
\end{abstract}

\thanks{The author is an associate member of the Hausdorff Center for Mathematics at the University of Bonn (DFG GZ 2047/1, project ID 390685813).}
\maketitle 

\section{Introduction}
Nowadays, there are two approaches to $G$-equivariant stable homotopy theory for a fixed finite group $G$, peacefully coexisting despite fundamental philosophical differences. The classical approach, going back to the pioneering work of Lewis, May, and Steinberger \cite{LMS}, uses $G$-equivariant versions of classical pointset models of spectra, viewed through a suitable notion of \emph{$G$-equivariant weak equivalences}; we will refer to the $\infty$-category $\mySp_G$ arising from any of these models as the \emph{$\infty$-category of (genuine) $G$-spectra} below. On the other hand, in more recent years a purely $\infty$-categorical approach has risen to prominence, based on the idea that we can encode $G$-equivariant objects in terms of their \emph{genuine fixed points} for all subgroups $H\subset G$ together with suitable structure maps relating them. Unstably, this idea is not at all new, and is put on a rigorous footing by Elmendorf's Theorem \cite{elmendorf}, which in modern terminology says that the (genuine) homotopy theory of topological spaces with $G$-action can be described as the $\infty$-category $\Fun^\times(\mathbb F_G^\op,\Spc)$ of product-preserving presheaves on the 1-category $\mathbb F_G$ of finite $G$-sets. A first stable version of Elmendorf's result was obtained by Guillou and May \cite{guillou-may-ps}, who showed that the $\infty$-category $\mySp_G$ is equivalent to product-preserving spectral presheaves on (the local group completion of) a certain handcrafted $(2,1)$-category of \emph{spans} of finite $G$-sets. Shortly after Guillou and May first announced their result, Barwick \cite{barwick2017spectral} developed a general theory allowing to build an $\infty$-category $\Span(\Cc)$ of spans in any suitable $\infty$-category $\Cc$, leading to a general theory of \emph{spectral Mackey functors}. Specializing this to $\Cc=\mathbb F_G$ yields an $\infty$-category $\Span(\mathbb F_G)$, which we can then use to define the $\infty$-category
\[
	\Mack_G(\Sp)\coloneqq\Fun^\times(\Span(\mathbb F_G),\Sp)\simeq\Fun^\times(\Span(\mathbb F_G)^\op,\Sp)
\]
of \emph{$G$-equivariant spectral Mackey functors}. While it isn't clear a priori that this construction agrees with the one of Guillou--May, Clausen--Mathew--Nauman--Noel \cite{cmnn} showed that we indeed have a preferred equivalence $\Mack_G(\Sp)\simeq\mySp_G$.

\subsection*{Equivariant algebraic $\bm K$-theory}
For both of the above approaches, equivariant versions of (group completion) \emph{algebraic $K$-theory} have been developed.
For the classical approach, there are two (equivalent) constructions due to Guillou--May \cite{guillou-may} and Merling \cite{merling}, taking a symmetric monoidal 1-category with $G$-action $\Cc$ as input, and then proceeding via an explicit construction on the level of models to produce a genuine $G$-spectrum $\mathbb K_G(\Cc)$.
Using the higher categorical description of equivariant stable homotopy theory, on the other hand, it is natural to start with a $G$-equivariant Mackey functor in symmetric monoidal $\infty$-categories, and then simply apply the classical non-equivariant $K$-theory construction pointwise. If we are instead given a plain symmetric monoidal (1- or $\infty$-)category with $G$-action, we can produce such a categorical $G$-equivariant Mackey functor via the right adjoint $\smash{(-)^\flat_G}$ to the forgetful functor $\smash{\Fun^\times(\Span(\mathbb F_G),\SymMonCat_\infty)\to\Fun(BG,\SymMonCat_\infty)}$ (evaluating at the free $G$-orbit). In the case of a symmetric monoidal $1$-category $\Cc$, one can also describe $\Cc^\flat_G$ quite explicitly in more classical terms: it records the categories of \emph{homotopy fixed points} for subgroups $H\subset G$ together with  the evident restriction maps for subgroup inclusions as well as the usual \emph{symmetric monoidal norms} \cite[Appendix A]{CHLL_NRings}. Note that this construction can also be easily adapted to other flavors of $K$-theory like Waldhausen $K$-theory or the $K$-theory of stable $\infty$-categories, and various such versions have appeared in the literature \cite{spectral-mackey-two, hilman2022parametrised, Elmanto_Haugseng_Bispans}.

While the above two definitions of equivariant algebraic $K$-theory are a priori unrelated, we prove in this paper:

\begin{introthm}[See Theorem~\ref{thm:equiv-comp} for a precise statement]\label{introthm:equivariant}
	The diagram
	\[
		\begin{tikzcd}
			\Fun(BG,\SymMonCat_1)\arrow[d,"\mathbb K_G"']\arrow[r, "(-)^\flat_G"] &[.25em] \Mack_G(\SymMonCat_\infty)\arrow[d, "\mathbb K\circ-"]\\
			\mySp_G\arrow[r,"\sim"'] & \Mack_G(\Sp)
		\end{tikzcd}
	\]
	commutes up to preferred equivalence.
\end{introthm}

We moreover prove a version of this theorem for \emph{global algebraic $K$-theory} in the sense of \cite{schwede-k}, which produces so-called \emph{global spectra} from symmetric monoidal 1-categories. While the $\infty$-category $\mySp_\gl$ of global spectra was again originally defined in terms of a pointset model, there is a notion of \emph{global Mackey functors} and an equivalence $\mySp_\gl\iso\Mack_\gl(\Sp)$ \cite{global-mackey,CLL_Spans}. Analogously to the equivariant case, we have a \emph{global Borel construction} $(-)^\flat\colon\SymMonCat_\infty\to\Mack_\gl(\SymMonCat_\infty)$ \cite{Elmanto_Haugseng_Bispans,puetzstueck-new} right adjoint to the forgetful functor, and we show:

\begin{introthm}
	[See Theorem~\ref{thm:main-comp} for a precise statement]\label{introthm:global}
	The diagram
	\begin{equation}\tag{$*$}
		\begin{tikzcd}
			\SymMonCat_1\arrow[d,"\mathbb K_\gl"']\arrow[r, "(-)^\flat"] &[.25em] \Mack_\gl(\SymMonCat_\infty)\arrow[d, "\mathbb K\circ-"]\\
			\mySp_\gl\arrow[r,"\sim"'] & \Mack_\gl(\Sp)
		\end{tikzcd}
	\end{equation}
	commutes up to preferred equivalence.
\end{introthm}

Also the global Mackey functor $\Cc^\flat$ can be described rather explicitly for a symmetric monoidal $1$-category $\Cc$ \cite[Example~3.1]{puetzstueck-new}: it coherently organizes the categories of $G$-objects in $\Cc$ for all finite groups $G$ together with the evident restrictions and with symmetric monoidal norms defined as in the equivariant setting. Taking $K$-groups thus results in what is sometimes called the \emph{Swan $K$-theory} of $\Cc$.

\subsection*{How to prove (and how not to prove) Theorems~\ref{introthm:equivariant} and~\ref{introthm:global}} 
It is not hard to show that the above diagrams commute after postcomposing with any of the evaluation functors $\Mack_\gl(\Sp)\to\Sp$ or $\Mack_G(\Sp)\to\Sp$; namely, the bottom equivalence is in either case given pointwise by forming the appropriate \emph{(categorical) fixed point spectra}, and the fixed point spectra of the classical equivariant and global algebraic $K$-theory constructions have already been described in the literature. However, producing a comparison that also takes the structure maps encoded in a spectral $G$-Mackey functor into account is a whole different beast: for example, already \cite[Theorem~6.21]{schwede-k}, which essentially states that the diagram from Theorem~\ref{introthm:global} commutes after postcomposing with $\pi_0\colon\Mack_\gl(\Sp)\to\Mack_\gl(\text{Ab})$, is computationally quite involved. One key problem here is that the spectral Mackey functors associated to equivariant or global spectra in particular record the \emph{homotopy theoretic transfers}, which are defined geometrically in terms of Thom--Pontryagin collapse maps, and it is not at all clear why they should have anything to do with the purely categorically defined symmetric monoidal norms. Given these difficulties already at the level of $\pi_0$, proving the above \emph{fully coherent} comparisons by hand seems completely infeasible, and we have to look for a more abstract approach.

To motivate our strategy, let us take a step back and look at the non-equivariant algebraic $K$-theory of symmetric monoidal $1$-categories. Using $\infty$-categorical language throughout, we can break this up into two steps: the first one takes a symmetric monoidal 1-category, i.e.~an $E_\infty$-monoid in $\Cat_1$, and forms its \emph{classifying space}, yielding an $E_\infty$-monoid in $\Spc$; in the second step, we then form the group completion of this $E_\infty$-monoid and consider it as a (connective) spectrum. It is a somewhat surprising theorem of Mandell \cite{mandell-Gamma} that the first step in this process actually exhibits the $\infty$-category $\CMon$ of $E_\infty$-monoids as a localization of $\SymMonCat_1$, namely at those maps that induce equivalences of classifying spaces. Combining this with the universal property of $\CMon$ as the free presentable semiadditive $\infty$-category \cite{groth-gepner-nikolaus}, one can then uniquely characterize $\mathbb K\colon\SymMonCat_1\to\Spc$ by the following properties:
\begin{enumerate}
	\item $\mathbb K$ inverts all maps that induce equivalences of classifying spaces, and the induced functor from the corresponding localization to $\Sp$ is a left adjoint.
	\item $\mathbb K$ sends the symmetric monoidal groupoid of finite sets under disjoint union to the sphere spectrum (the \emph{Barratt--Priddy--Quillen Theorem}).
\end{enumerate}
Our proof relies on similar descriptions of the pointset constructions of equivariant and global algebraic $K$-theory. However, while \cite{g-global} provides the relevant equivariant and global versions of Mandell's result, there is a problem: the $\infty$-categories of (genuine) $G$-equivariant or global $E_\infty$-monoids on their own do \emph{not} come with a nice universal property.

In order to fix this, we move to the setting of \emph{parametrized higher category theory} in the sense of \cite{exposeI}. By \cite{CLL_Clefts}, we can assemble the $\infty$-categories of $G$-equivariant $E_\infty$-monoids for all finite groups $G$ together with the restriction functors between them into a \emph{global $\infty$-category}, i.e.\ a product-preserving functor $\Fglo^\op\to\Cat_\infty$, where $\Fglo$ denotes the $\infty$-category of finite $1$-groupoids, and this global $\infty$-category admits a universal property analogous to the one for $\CMon$. In the global case, we have to more generally consider the \emph{$G$-global $E_\infty$-monoids} from \cite{g-global}; varying $G$ these again assemble into a global $\infty$-category $\ul\GammaS_\gl^\text{spc}$ satisfying a similar universal property \cite{CLL_Global}. This way, we can for example characterize the family of equivariant algebraic $K$-theory functors with $G$ varying through all finite groups in an analogous manner to the characterization of non-equivariant $K$-theory above, allowing us to completely forget about its concrete pointset construction. Another advantage of the parametrized perspective is that there are also universal properties for the global $\infty$-categories of $G$-equivariant or $G$-global spectra, which in particular provides uniqueness results for the equivalences $\smash{\mySp_G\iso\Mack_G(\Sp)}$ and $\smash{\mySp_\text{$G$-gl}\iso\Mack_\text{$G$-gl}(\Sp)}$ as soon as we require them to vary naturally in $G$, releasing us from the burden to justify that the unnamed equivalences in Theorems~\ref{introthm:equivariant} and~\ref{introthm:global} are the `correct' ones.

Note that if we want to prove Theorem~\ref{introthm:global} on global algebraic $K$-theory along the above lines, we are forced to consider its generalization to \emph{$G$-global algebraic $K$-theory} in the sense of \cite{g-global}; on the other hand, this $G$-global comparison will turn out to also imply Theorem~\ref{introthm:equivariant} rather easily. Most of the present paper is thus devoted to proving commutativity of a certain diagram of \emph{global functors} (i.e.\ natural transformations of functors $\Fglo^\op\to\Cat_\infty$) 
\begin{equation*}
	\begin{tikzcd}
		\SymMonCat_1^\flat\arrow[r]\arrow[d,"\mathbb K_\gl"'] & \ul\Mack_\gl(\SymMonCat_\infty)\arrow[d,"\mathbb K\circ-"]\\
		\ul\mySp_\gl\arrow[r,"\sim"'] & \ul\Mack_\gl(\Sp)
	\end{tikzcd}
\end{equation*}
lifting the square $(*)$ from Theorem~\ref{introthm:global}; here $\SymMonCat_1^\flat$ denotes the global category sending a finite group $G$ to the $\infty$-category of symmetric monoidal $1$-categories with $G$-action, with the evident restrictions. By the results from \cite{g-global}, we know that the bottom path through this diagram descends to a left adjoint global functor after suitably localizing the top left corner, so if we could show that also the upper path descends to a left adjoint, it would be enough by the universal property to chase through a single object. While one could indeed prove this adjointness statement directly, it will be more convenient to split the comparison into two steps: analogously to its non-equivariant counterpart, the $G$-global algebraic $K$-theory functor can be obtained as a composition of a functor building a $G$-global $E_\infty$-monoid from a symmetric monoidal category with $G$-action, followed by a delooping functor producing a (connective) $G$-global spectrum from this. This yields a factorization of the global functor $\mathbb K_\gl$ through the global $\infty$-category $\ul\GammaS_\gl^\text{spc}$, and the technical heart of our argument is Theorem~\ref{thm:swan-equivalence}, which establishes commutativity of the analogous diagram
\begin{equation}\tag{$\dagger$}
	\begin{tikzcd}
		\SymMonCat_1^\flat\arrow[r]\arrow[d] & \ul\Mack_\gl(\SymMonCat_\infty)\arrow[d]\\
		\ul\GammaS_\gl^\text{spc}\arrow[r,"\sim"'] & \ul\Mack_\gl(\Spc)\rlap,
	\end{tikzcd}
\end{equation}
where the vertical map on the left now even descends to an \emph{equivalence}. Instead of showing that the upper path induces a left adjoint, it will then be easier to show that it induces a \emph{right} adjoint. Upon passing to adjoints everywhere, the universal property of $\ul\GammaS_\gl^\text{spc}$ then translates to saying that we can check commutativity of $(\dagger)$ after postcomposing with the forgetful functor $\ul\Mack_\gl(\Spc)\to\Spc$, which as mentioned before is actually straightforward.

\subsection*{Outline} In Section~\ref{sec:prelim} we recall the necessary pointset models and Mackey functor models of $G$-global homotopy theory along with the construction of $G$-global algebraic $K$-theory. We moreover introduce the requisite background from parametrized higher category theory to state the relevant universal properties from \cite{CLL_Global}.

Section~\ref{sec:heart} constitutes the technical heart of the paper: we construct the global functor $\SymMonCat_1^\flat\to\ul\Mack_\gl(\SymMonCat_\infty)$ lifting the global Borel construction $(-)^\flat\colon\SymMonCat_1\to\Mack_\gl(\SymMonCat_\infty)$, and we show that the diagram $(\dagger)$ commutes, following the strategy outlined above.

Finally, we deduce both Theorems~\ref{introthm:equivariant} and~\ref{introthm:global} from this comparison in Section~\ref{sec:main-results}.

\subsection*{Historical remark}
I originally proved a variant of Theorem~\ref{introthm:global} as part of the first version of the arXiv preprint \cite{global-mackey} (back then titled `Global algebraic $K$-theory is Swan $K$-theory'); this used a handcrafted equivalence between global spectra and global spectral Mackey functors, produced in the same paper. While my original argument already used the global version of Mandell's theorem to reduce the comparison of $K$-theory constructions to a more manageable subcategory of $\SymMonCat_1$, the parametrized results used crucially in the present paper were not yet available at that time, and so I ultimately had to resort to some rather technical, 2-categorical computations. As the parametrized approach allows for a much cleaner proof of a more general comparison result, I have decided to remove the original version of Theorem~\ref{introthm:global} from \cite{global-mackey} and instead devote the present paper to the proof of the more general comparison.

\subsection*{Notational convention} Throughout, we will denote $\infty$-categories in ordinary upright font---for example, we write $\Cat_\infty$ for the $\infty$-category of small $\infty$-categories, and $\Cat_1\subset\Cat_\infty$ for the full subcategory of $1$-truncated $\infty$-categories. We will distinguish $1$-categorical models of $\infty$-categories by writing them in boldface font, so that for example $\cato{Cat}$ denotes the $1$-category of small $1$-categories (which comes with a functor $\cato{Cat}\to\Cat_1$ that is a Dwyer--Kan localization at the equivalences of categories).

\subsection*{Acknowledgements} I am grateful to Phil Pützstück for helpful comments on an earlier version of this paper. It is moreover a pleasure to thank Bastiaan Cnossen and Sil Linskens for the collaboration on all things parametrized, without which this article would not have been possible.

\section{Global homotopy theory and global category theory}\label{sec:prelim}
In this section we will introduce our reference models of global homotopy theory in the sense of \cite{schwede2018global,hausmann-global}, and more generally of \emph{$G$-global homotopy theory} in the sense of \cite{g-global} for any finite group $G$. While we will be ultimately interested in ($G$-)global versions of spectra, the appropriate notions of spaces and $E_\infty$-monoids will play an important role in both the construction of global algebraic $K$-theory as well as in the proof of our main results. 

\subsection{The unstable story} We therefore begin with our pointset model of unstable $G$-global homotopy theory, which relies on a certain simplicial monoid $E\mathcal M$ typically referred to as the \emph{universal finite group} in this context:

\begin{construction}
	Write $\mathcal M$ for the (discrete) monoid of self-injections of the countably infinite set $\omega=\{0,1,2,\dots\}$, with monoid structure given by composition. 

	We now consider the right adjoint $E\colon\cat{Set}\to\cat{Cat}_1$ of the functor sending a small $1$-category to its set of objects; applying this to the above monoid $\Mm$, we obtain a categorical monoid (i.e.~strict monoidal category) $\mathcal M$. We will also write $E\mathcal M$ for the simplicial monoid obtained as the nerve of this; explicitly, this is the simplicial set satisfying $(E\mathcal M)_n=\hom_{\cato{Set}}([n],\mathcal M)$ with the evident simplicial structure maps and with pointwise multiplication.
\end{construction}

Note that $E\mathcal M$ is of course neither finite nor a group; nevertheless, the moniker `universal finite group' is justified by the fact that we can embed any finite group in a canonical way into $E\mathcal M$:

\begin{definition}
	A finite subgroup $H\subset\mathcal M$ is called \emph{universal} if the restriction of the tautological $\Mm$-action on $\omega$ turns the latter into a \emph{complete $H$-set universe}, meaning that every finite $H$-set embeds (non-uniquely) into it.
\end{definition}

\begin{lemma}[See {\cite[Lemma~1.2.8]{g-global}}]
	Let $H$ be a finite group. Then there exists an injective homomorphism $H\to\Mm$ whose image is a universal subgroup. Moreover, any two such homomorphisms agree up to conjugation by an invertible element of $\Mm$.\qed
\end{lemma}

As a consequence, an object with an $E\mathcal M$-action in a complete simplicial category has fixed points for every abstract finite group, as one would expect from a global object. And indeed, while an $E\mathcal M$-action in the homotopy coherent sense does not contain any extra information (the simplicial set $E\mathcal M$ is weakly contractible), $E\mathcal M$-actions on the pointset level turn out to conveniently encode global homotopy theoretic information:

\begin{construction}
	Fix a finite group $G$, and let $X\in\cato{$\bm{E\mathcal M}$-$\bm G$-SSet}$, be an $E\mathcal M$-$G$-simplical set, i.e.~a simplicial set with a (strict) action by the simplicial monoid $E\mathcal M\times G$. For every universal subgroup $H$ of $\mathcal M$ and every homomorphism $\phi\colon H\to G$ we will write $X^\phi$ for the (1-categorical) fixed points with respect to the \emph{graph subgroup} $\Gamma_{H,\phi}\coloneqq\{(h,\phi(h)):h\in H\}\subset E\mathcal M\times G$.

	We then define the class $\Ww_\text{$G$-gl}$ of \emph{$G$-global weak equivalences} to consist of all those maps $f\colon X\to Y$ in $\cato{$\bm{E\mathcal M}$-$\bm{G}$-SSet}$ such that $f^\phi\colon X^\phi\to Y^\phi$ is a weak homotopy equivalence for all universal $H$ and all $\phi\colon H\to G$. The $\infty$-category $\S_{\text{$G$-gl}}$ of \emph{$G$-global spaces}  is the Dwyer--Kan localization of the 1-category $\cato{$\bm{E\mathcal M}$-$\bm G$-SSet}$ at the $G$-global weak equivalences.
\end{construction}

By \cite[Corollary~1.2.34]{g-global}, the $G$-global weak equivalences participate in a combinatorial model structure; in particular, $\S_\text{$G$-gl}$ is a presentable $\infty$-category.

\begin{remark}
	While the above model is often the most convenient one, there are various alternative model categories of $G$-global spaces, each (Quillen) equivalent to it, see \cite[Chapter 1]{g-global}. One of these alternative models will play a role in this paper, though only as a blackbox:

	Let us write $I$ for the 1-category of finite sets and injections, and $\mathcal I$ for the simplicial category obtained by applying $E$ to all hom sets of $I$. Then there are various model structures on the 1-category $\cato{$\bm{G}$-$\bm{\mathcal I}$-SSet}$ of \emph{$G$-$\mathcal I$-simplicial sets} \cite[Section~1.4]{g-global} (i.e.~$G$-objects in the 1-category of enriched functors $\mathcal I\to\cato{SSet}$) all sharing the same class of weak equivalences and modelling the $\infty$-category $\S_\text{$G$-gl}$.
\end{remark}

In addition to these pointset models, there is also a purely $\infty$-categorical description of $G$-global spaces as a suitable presheaf category; we can view this comparison as a $G$-global version of Elmendorf's theorem \cite{elmendorf}. Stating the strongest version of this result will require some parametrized higher category theory, and we therefore begin with the necessary setup.

\begin{definition}
	We write $\Fglo$ for the 2-category of finite 1-groupoids, and we write $\Glo\subset\Fglo$ for the full subcategory spanned by the connected (and hence in particular non-empty) finite 1-groupoids.
\end{definition}

We will conflate a finite group $G$ with the corresponding 1-object groupoid (otherwise often denoted $BG$). In this interpretation, the objects of $\Glo$ are given by finite groups, its $1$-morphisms by group homomorphisms, and it has a $2$-cell $\phi\Rightarrow g\cdot \phi(-)\cdot g^{-1}$ for every homomorphism $\phi\colon H\to G$ and every $g\in G$.

\begin{definition}
	A \emph{global $\infty$-category} is a functor $\Glo^\op\to\CAT_\infty$ into the large $\infty$-category of $\infty$-categories. We write $\CAT_{\gl,\infty}\coloneqq\Fun(\Glo^\op,\CAT_\infty)$ for the large $\infty$-category of global $\infty$-categories. 
\end{definition}

As the inclusion $\Glo\hookrightarrow\Fglo$ exhibits the target as the free finite coproduct completion of the source, we may equivalently view a global $\infty$-category as a product preserving functor $\Fglo^\op\to\Cat_\infty$. In particular, while we will often only specify the values of a global $\infty$-category on $\Glo$, we will nevertheless plug in general finite 1-groupoids; whenever we do so, we have tacitly passed to the unique product preserving extension (i.e.~the right Kan extension along $\Glo^\op\hookrightarrow\Fglo^\op$).

\begin{example}\label{ex:Borel}
	Let $\Cc$ be any $\infty$-category. Then we obtain a \emph{global Borel $\infty$-category} $\Cc^\flat$ as the composite
	\[
		\Fglo^\op\hookrightarrow\Cat_\infty^\op\xrightarrow{\;\Fun(-,\Cc)\;}\CAT_\infty.
	\]
	Note that the inclusion $\Fglo\hookrightarrow\Spc$ is the pointwise left Kan extension of the inclusion of the terminal object (as is the inclusion of \emph{any} full subcategory containing the terminal object, by the co-Yoneda Lemma), whence so is $\Fglo\hookrightarrow\Cat$ since $\Spc\hookrightarrow\Cat$ is cocontinuous. We conclude that $\Cc^\flat$ is right Kan extended from the inclusion $1\hookrightarrow\Fglo$, and that $(-)^\flat\colon\CAT_\infty\to\CAT_{\gl,\infty}$ is right adjoint to the functor $\ev_{1}\colon\CAT_{\gl,\infty}\to\CAT_\infty$ sending a global $\infty$-category $\Cc$ to its \emph{underlying $\infty$-category} $\Cc(1)$.
\end{example}

\begin{example}\label{ex:Sgl}
	Specializing the previous example, we obtain a global $\infty$-category $\cato{$\bm{E\mathcal M}$-SSet}^\flat$ sending a finite group $G$ to $\cato{$\bm{E\mathcal M}$-$\bm G$-SSet}$ and a group homomorphism $f\colon H\to G$ to the evident restriction functor $f^*$. It is clear from the definitions that $f^*$ sends $G$-global weak equivalences to $H$-global ones; thus, passing to the appropriate Dwyer--Kan localizations we obtain a global $\infty$-category $\ul\S_\gl\coloneqq\cato{$\bm{E\mathcal M}$-SSet}^\flat[\Ww_\gl^{-1}]$ sending a finite group $G$ to the $\infty$-category of $\S_\text{$G$-gl}$ of $G$-global spaces, and a homomorphism $f\colon H\to G$ to the total derived functor of the pointset level restriction functor.
\end{example}

There is also a second, more `genuine' way to build a global $\infty$-category from a plain $\infty$-category, and this will yield the presheaf model of $G$-global spaces:

\begin{example}\label{ex:sub-gl}
	We write $\Glo_{/-}\colon\Glo\to\Cat$ for the functor sending $G\in\Glo$ to the slice $\Glo_{/G}$, with functoriality given by postcomposition; more formally, this is the cocartesian straightening of the target map $\ev_1\colon\Ar(\Glo)\to\Glo$. Given any $\infty$-category $\Cc$, we then define the global $\infty$-category $\ul{\Cc}_\gl$ as the composite
	\[
		\Glo^\op\xrightarrow{\;(\Glo_{/-})^\op\;}\Cat_\infty^\op\xrightarrow{\;\Fun(-,\Cc)\;}\CAT_\infty.
	\]
\end{example}

\begin{remark}\label{rk:alternative-sub-gl}
	If $\Cc$ has finite products, we can equivalently describe the above as the global $\infty$-category $\Fglo^\op\to\CAT_\infty$ sending $G\in\Fglo$ to $\Fun^\times(\Fglo_{/G}^\op,\Cc)$, with the evident functoriality.
\end{remark}

With this at hand, we can now state our comparison between the global and model categorical approaches to unstable $G$-global homotopy theory:

\begin{theorem}[See {\cite[Theorem~3.3.1]{CLL_Global}}]\label{thm:unique-equiv-unstable}
	There exists a unique equivalence $\ul\Spc_\gl\simeq\ul{\S}_\gl$ of global $\infty$-categories.\qed
\end{theorem}

In fact, the global $\infty$-category $\ul{\Spc}_\gl$ has a universal property, and hence so has $\ul{\S}_\gl$. In order to state this property, we first need a couple of further definitions.

We begin by observing that the $\infty$-category $\CAT_{\gl,\infty}$ of global $\infty$-categories is cartesian closed since $\CAT_\infty$ is so. We write $\ul\Fun_\gl(\Cc,\Dd)$ for the internal hom, and $\Fun_\gl(\Cc,\Dd)\coloneqq\ul\Fun_\gl(\Cc,\Dd)(1)$ for its underlying $\infty$-category. This way, one could upgrade $\CAT_{\gl,\infty}$ to a large $(\infty,2)$-category; all that we will need below, however, is that this allows us to define the homotopy $(2,2)$-category of global $\infty$-categories, and hence it in particular makes sense to talk about \emph{adjunctions} of global $\infty$-categories. Throughout this paper, we will use the following `pointwise' characterization of adjunctions:

\begin{lemma}[See {\cite[Proposition~3.2.9 and Corollary~3.2.11]{martiniwolf2021limits}}]\label{lemma:adjunctions-criterion}
	Let $F\colon\Cc\to\Dd$ be a global functor. Then $F$ is a left adjoint if and only if the following two conditions are satisfied:
	\begin{enumerate}
		\item For each $G\in\Glo$, the functor $F_G\colon \Cc(G)\to\Dd(G)$ admits a right adjoint $U_G$.
		\item For each $f\colon G\to H$ in $\Glo$, the Beck--Chevalley map\footnote{We refer the reader to \cite[Appendix~C]{CLL_Adams} for background material about Beck--Chevalley transformations.} $f^*U_H\to U_Gf^*$ associated to the naturality square
		\[
			\begin{tikzcd}
				\Cc(H)\arrow[d,"f^*"']\arrow[r,"F_H"] & \Dd(H)\arrow[d,"f^*"]\\
				\Cc(G)\arrow[r, "F_G"'] & \Dd(G)
			\end{tikzcd}
		\]
		is an equivalence.
	\end{enumerate}
	Moreover, in this case condition (1) holds more generally for all $G\in\Fglo$, and condition (2) similarly holds for all $f$ in $\Fglo$. The right adjoint $U$ is given pointwise for each $G\in\Fglo$ by $U_G$.
	
	The dual characterization of parametrized right adjoints holds.\qed
\end{lemma}

Assume now that the restrictions $f^*\colon\Cc(H)\to\Cc(G)$ and $\Dd(H)\to\Dd(G)$ admit left adjoints $f_!$ for every $f\in\Fglo$. Passing to total mates, we see that the Beck--Chevalley transformation $f^*U_H\to U_Gf^*$ from Condition (2) above is invertible if and only if the Beck--Chevalley transformation $f_!F_G\to F_Hf_!$ is so. This condition is often more convenient to check, and it moreover admits a conceptual interpretation as a parametrized version of coproduct preservation:

\begin{definition}
	Let $\Cc$ be a global category. We say that $\Cc$ \emph{has global coproducts} if the functor $f^*\colon\Cc(H)\to\Cc(G)$ admits a left adjoint for every $f\colon G\to H$ in $\Fglo$, and if moreover for every pullback
	\begin{equation}\label{diag:bc-pb}
		\begin{tikzcd}
			G'\arrow[r,"f'"]\arrow[dr,phantom,"\lrcorner"{very near start}]\arrow[d,"g"'] & H'\arrow[d,"h"]\\
			G\arrow[r, "f"'] & H
		\end{tikzcd}
	\end{equation}
	in $\Fglo$, the Beck--Chevalley transformation $f'_!g^*\to h^*f_!$ is invertible. If also $\Dd$ has global coproducts, then we say that $F\colon\Cc\to\Dd$ \emph{preserves global coproducts} if the Beck--Chevalley map $f_!F_G\to F_Hf_!$ is invertible for every $f\colon G\to H$ in $\Fglo$.

	Dually, we define what it means for $\Cc$ and $\Dd$ to have \emph{global products} and for $F$ to \emph{preserve global products}.
\end{definition}

\begin{variant}
	If in the above situation we instead only require the existence of adjoints and the basechange formula in the case where $f$ (and hence also $f'$) is \emph{faithful}, we say that $\Cc$ has \emph{equivariant (co)products}. The definition of what it means to \emph{preserve equivariant (co)products} should then be obvious.
\end{variant}

\begin{remark}
	In contrast to the characterization of adjunctions, it is important in the above definition that we consider all maps in $\Fglo$, and not just in $\Glo$; for example, the existence of the left adjoints for fold maps $\nabla\coloneqq(\id,\id)\colon G\amalg G\to G$ yields the existence of ordinary (finite) coproducts for each $\Cc(G)$, while the basechange condition amounts to demanding that each of the restriction functors $f^*\colon\Cc(G)\to\Cc(H)$ preserve these coproducts \cite[Proposition~4.2.14]{CLL_Global}.
\end{remark}

\begin{remark}
	Again, there is also an `internal' characterization of global and equivariant (co)products in terms of adjunctions of parametrized functor categories; we will not need this alternative perspective, but refer the curious reader to \cite[Example~4.1.3, Remark~4.1.5, and Proposition~5.4.2]{martiniwolf2021limits}.
\end{remark}

\begin{example}\label{ex:Borel-prod}
	Let $\Cc$ be a cocomplete $\infty$-category. Then $\Cc^\flat$ has global coproducts \cite[Example~2.15]{CLL_Spans}. Dually, if $\Cc$ is complete, then $\Cc^\flat\simeq \big((\Cc^\op)^\flat\big)^\op$ has global products.
\end{example}

\begin{example}[cf.~{\cite[Lemma~8.13]{CLL_Spans}}]\label{ex:gl-glproducts}
	Let $\Cc$ be an $\infty$-category with finite products. Then $\ul\Cc_\gl$ has global products: using the alternative description from Remark~\ref{rk:alternative-sub-gl}, we note that each $f_!\colon\Fglo_{/G}\to\Fglo_{/H}$ admits a coproduct preserving right adjoint $f^*$ (given by pullback), so that the restriction functor $\Fun^\times(\Fglo_{/H}^\op,\Cc)\to\Fun^\times(\Fglo_{/G}^\op,\Cc)$ has a right adjoint given by restriction along $f^*$ thanks to 2-functoriality of $\Fun^\times(-,\Cc)$. By the same argument, it will suffice for the Beck--Chevalley condition that the Beck--Chevalley map $g_!f^{\prime*}\to f^*h_!$ (of functors $\Fglo_{/H'}\to \Fglo_{/G}$) associated to
	\[
		\begin{tikzcd}
			\Fglo_{/G'}\arrow[r,"f'_!"]\arrow[d,"g_!"'] & \Fglo_{/H'}\arrow[d,"h_!"]\\
			\Fglo_{/G}\arrow[r, "f_!"'] & \Fglo_{/H}
		\end{tikzcd}
	\]
	is invertible for every pullback $(\ref{diag:bc-pb})$, which is just a reformulation of the pasting law for pullbacks.
\end{example}

\begin{definition}
	A global $\infty$-category $\Cc\colon\Glo^\op\to\CAT_\infty$ is \emph{presentable} if it has global coproducts and is moreover \emph{fiberwise presentable} in the sense that it factors through the non-full subcategory $\Pr^\textup{L}\subset\CAT_\infty$ of presentable $\infty$-categories and left adjoint functors.
\end{definition}

With this terminology at hand, we can now finally state the universal property of $\ul\Spc_\gl$:

\begin{theorem}[See {\cite[Corollary~2.4.11]{CLL_Global}} or {\cite[Theorem~7.1.1]{martiniwolf2021limits}}]\label{thm:univ-prop-spcgl}
	The global $\infty$-category $\ul\Spc_\gl$ is globally presentable. Moreover, if $\Dd$ is any other presentable global $\infty$-category, then evaluation at the terminal object $1\in\Spc_\gl\coloneqq\ul\Spc_\gl(1)$ determines an equivalence
	\[
		\Fun^\textup{L}_\gl(\ul\Spc_\gl,\Dd)\iso\Dd(1),
	\]
	where the left hand side denotes the full subcategory of left adjoint functors.\qed
\end{theorem}

We close this subsection by giving a pointwise description of the unique equivalence $\Phi\colon\ul\S_\gl\iso\ul\Spc_\gl$:

\begin{remark}\label{rk:unstable-pw-desc}
	Let $H\subset\Mm$ be a universal subgroup. We claim that for every $\phi\colon H\to G$ in $\Glo$ the composite
	\[
		\S_\text{$G$-gl}\xrightarrow[\smash{\raise2pt\hbox{$\scriptstyle\sim$}}]{\;\Phi_G\;} \Fun(\Glo_{/G}^\op,\Spc)\xrightarrow{\;\ev_\phi\;}\Spc
	\]
	agrees with the derived functor of $(-)^\phi\colon\cato{$\bm{E\mathcal M}$-$\bm G$-SSet}\to\cato{SSet}$. While this could be deduced by carefully looking at the concrete construction given in \cite[Subsection~3.3]{CLL_Global}, we provide a more abstract argument:

	For this, let us first look at the inverse equivalence $\Lambda\colon\ul\Spc_\gl\to\ul\S_\gl$. We consider the $E\mathcal M$-$H$-simplicial set $E\mathcal M$ with $H$ acting via multiplication on the right. By \cite[Example~1.2.35]{g-global}, this then models the terminal object of $\S_\text{$H$-gl}$, which in turn agrees with $\Lambda_H(1)$ as equivalences preserve terminal objects. On the other hand, $E\mathcal M$ also represents the functor $(-)^{\id_H}\colon\cato{$\bm{E\mathcal M}$-$\bm H$-SSet}\to\cato{SSet}$ in the simplicially enriched sense; as this functor is homotopical and since $E\mathcal M$ is cofibrant in the $H$-global model structure \cite[Corollary~1.2.34 and Example~1.2.35]{g-global}, we conclude that the terminal object $\Lambda_H(1)$ of $\S_\text{$H$-gl}$ represents the total derived functor of $(-)^{\id_H}$. Passing to right adjoints, then immediately implies the claim for $\phi=\id_H$. The general case follows from this as $\ev_\phi\Phi_G\simeq \ev_{\id_H}\phi^*\Phi_G\simeq\ev_{\id_H}\Phi_H\phi^*$ by definition of the functoriality of $\ul\Spc_\gl$ and naturality of $\Phi$, respectively.
\end{remark}

\subsection{Ultra-commutativity} Next, we turn to a (genuine) $G$-global version of \emph{$E_\infty$-monoids}. Our reference model is a certain refinement of Segal's description of classical $E_\infty$-monoids in terms of \emph{special $\varGamma$-spaces} \cite{segal} and Shimakawa's model of genuine $G$-$E_\infty$-monoids via \emph{special $\varGamma_G$-spaces} \cite{shimakawa,shimakawa-simplify}:

\begin{construction}
	We will denote the 1-category of finite pointed sets by $\Gamma$.\footnote{Beware that some authors use the notation $\Gamma^\op$ for this category.} For every $n\ge0$, we write $n^+$ for the set $\{0,1,\dots,n\}$ with basepoint 0. 
	
	We now consider the $1$-category $\cato{$\bm\Gamma$-$\bm{E\mathcal M}$-$\bm G$-SSet}_*$ of reduced functors $X\colon\Gamma\to \cato{$\bm{E\mathcal M}$-$\bm G$-SSet}$, i.e.~functors such that $X(0^+)$ is terminal. A map $f\colon X\to Y$ in $\cato{$\bm\Gamma$-$\bm{E\mathcal M}$-$\bm G$-SSet}_*$ is called a \emph{$G$-global level weak equivalence} if $f(n^+)\colon X(n^+)\to Y(n^+)$ is a $(\Sigma_n\times G)$-global weak equivalence for every $n\ge0$, where the symmetric group acts via the functoriality of $X$ and $Y$. The $\infty$-category $\GammaS_\text{$G$-gl}$ of \emph{$G$-global $\varGamma$-spaces} is then defined as the Dwyer--Kan localization of $\cato{$\bm\Gamma$-$\bm{E\mathcal M}$-$\bm G$-SSet}_*$ at the $G$-global level weak equivalences. We write $\GammaS_\text{$G$-gl}^\text{spc}$ for the full subcategory spanned by those $G$-global $\Gamma$-spaces that are \emph{special}, in the sense that the usual Segal map $X(n^+)\to\prod_{i=1}^n X(1^+)$ is a $(G\times\Sigma_n)$-global weak equivalence for every $n\ge 0$; here $\Sigma_n$ acts on the left hand side as before and on the right hand side via permuting the factors of the product.
\end{construction}

By \cite[Corollary~2.2.53 and Theorem~2.2.55]{g-global}, also $\GammaS_\text{$G$-gl}^\text{spc}$ is presentable.

\begin{remark}
	Again, there are various alternative models, see \cite[Chapter~2]{g-global}, \cite{g-gl-operad}, or \cite{starmod}, and in particular there is a model $\cato{$\bm\Gamma$-$\bm G$-$\bm{\mathcal I}$-SSet}_*$ in terms of $G$-$\mathcal I$-spaces, which will occur as a blackbox later. As we will recall in the next section, one can moreover model `$G$-global $E_\infty$-monoids' by plain old symmetric monoidal 1-categories with $G$-action, which will be one of the key ingredients for the proof of our main theorems.
\end{remark}

Analogously to Example~\ref{ex:Sgl}, we can build a global $\infty$-category $\ul\GammaS_\gl$ sending $G\in\Glo$ to $\GammaS_\text{$G$-gl}$ with the evident restriction functors. The full subcategories $\GammaS_\text{$G$-gl}^\text{spc}\subset\GammaS_\text{$G$-gl}$ then assemble into a global subcategory $\ul\GammaS_\gl^\text{spc}$, see \cite[Definitions~5.1.1 and~5.1.6]{CLL_Global}. Both of these come with a forgetful functor $\mathbb U$ to $\ul\S_\gl$, obtained by deriving $(\ev_{1^+})^\flat\colon\cato{$\bm\Gamma$-$\bm{E\mathcal M}$-SSet}_*^\flat\to\cato{$\bm{E\mathcal M}$-SSet}^\flat$.

Also the global $\infty$-category $\ul\GammaS_\gl^\text{spc}$ enjoys a universal property; this relies on the notion of \emph{equivariant semiadditivity}, which roughly demands that a global $\infty$-category have equivariant products and that the requisite right adjoints $f_*$ to restrictions should also simultaneously be \emph{left adjoint} to restrictions, in a specific way. For the purposes of our paper, we will treat the notion of equivariant semiadditivity as a blackbox; the curious reader is referred to \cite[Definition~4.5.1 and Example~4.5.2]{CLL_Global} for a precise definition. We then have:

\begin{theorem}[See {\cite[Theorem~5.3.5]{CLL_Global}}]\label{thm:pres-univ-prop-Gamma}
	The global $\infty$-category $\ul\GammaS_\gl^\textup{spc}$ is presentable and equivariantly semiadditive. For any other such $\Dd$, evaluation at $\mathbb P(1)\in\GammaS_\gl^\textup{spc}$ determines an equivalence
	\[
		\Fun_\gl^\textup{L}(\ul\GammaS_\gl^\textup{spc},\Dd)\iso\Dd(1),
	\]
	where $\mathbb P\colon\ul\S_\gl\to\ul\GammaS_\gl^\textup{spc}$ is left adjoint to the forgetful functor $\mathbb U$.\qed
\end{theorem}

\begin{remark}
	The object $\mathbb P(1)$ (the free `genuine global $E_\infty$-monoid') can also be described very explicitly in terms of the pointset model, see \cite[Proposition~4.2.1 or Theorem~4.2.22]{g-global}.
\end{remark}

Again, there is also an alternative purely $\infty$-categorical description of $\ul\GammaS_\gl^\text{spc}$; this relies on Barwick's \emph{span categories} \cite{barwick2017spectral}, which we will recall now.

\begin{definition}
	An \emph{adequate triple} consists of an $\infty$-category $\Cc$ together with two wide subcategories $\Cc_L,\Cc_R$ with the following property: for every (solid) cospan
	\[
		\begin{tikzcd}
			\cdot\arrow[r,dashed,"r'"]\arrow[dr,"\lrcorner"{very near start},phantom]\arrow[d,"\ell'"',dashed] & \cdot\arrow[d,"\ell"]\\
			\cdot\arrow[r,"r"'] &\cdot
		\end{tikzcd}
	\]
	in $\Cc$, where $\ell\in\Cc_L$ and $r\in\Cc_R$, the (dashed) pullback exists, the map $\ell'$ again belongs to $\Cc_L$, and $r'$ belongs to $\Cc_R$. A \emph{morphism of adequate triples} $(\Cc,\Cc_L,\Cc_R)\to(\Dd,\Dd_L,\Dd_R)$ is a functor $\Cc\to\Dd$ sending $\Cc_L$ to $\Dd_L$, $\Cc_R$ to $\Dd_R$, and preserving the above pullbacks.
\end{definition}

We refer the reader e.g.~to \cite[Definition~2.12]{HHLNa} for the formal construction of the functor $\Span\colon\AdTrip\to\Cat_\infty$ sending an adequate triple $(\Cc,\Cc_L,\Cc_R)$ to its \emph{span category} $\Span(\Cc,\Cc_L,\Cc_R)$ (originally called \emph{effective Burnside category} in \cite{barwick2017spectral}). Below we will only consider adequate triples where $\Cc_L=\Cc$, in which case we will abbreviate $\Span_R(\Cc)\coloneqq\Span(\Cc,\Cc,\Cc_R)$. Informally, this $\infty$-category has the same objects as $\Cc$, with morphisms from $x$ to $y$ given by spans $x\gets z\to y$ such that the right-pointing morphism belongs to $\Cc_R$; composition is given via pullback. The $\infty$-category $\Span_R(\Cc)$ comes with natural inclusions $\Cc^\op\hookrightarrow\Span_R(\Cc)$ and $\Cc_R\hookrightarrow\Span_R(\Cc)$ identifying the sources with the wide subcategories of left- and right-pointing morphisms, respectively, see \cite[Proposition~2.15]{HHLNa} and the discussion preceeding it.

\begin{construction}
	Recall the functor $\Fglo_{/-}\colon\Fglo\to\Cat_\infty$ from Example~\ref{ex:sub-gl}. We write $\Forb\subset\Fglo$ for the wide subcategory of faithful functors, and consider the subfunctor $\Fglo_{/-}\times_{\Fglo}\Forb$. Then $(\Fglo_{/G},\Fglo_{/G},\Fglo_{/G}\times_{\Fglo}\Forb)$ is an adequate triple for every $G\in\Fglo$ as faithful functors are stable under pullback. As moreover each postcomposition functor $\Fglo_{/G}\to\Fglo_{/H}$ preserves pullbacks, we have therefore constructed a functor $\Fglo\to\text{AdTrip}$; passing to spans, we then obtain $\Span_{\Forb}(\Fglo_{/-})\colon\Fglo\to\Cat_\infty$.

	By e.g.~\cite[Proposition~2.2.5 and Example~3.1.9]{CHLL_NRings}, $\Span_{\Forb}(\Fglo_{/G})$ is semiadditive, and the inclusion $(\Fglo_{/G})^\op\hookrightarrow\Span_{\Forb}(\Fglo_{/G})$ preserves finite products.
	Given any $\infty$-category $\Cc$ with finite products, we can therefore define the global $\infty$-category $\ul\Mack_\gl(\Cc)$ as the composite
	\[
		\Fglo^\op\xrightarrow{\;\Span_{\Forb}(\Fglo_{/-})\;}(\Cat_\infty^\times)^\op\xrightarrow{\;\Fun^\times(-,\Cc)\;}\CAT_\infty.		
	\]
\end{construction}

\begin{theorem}\label{thm:canonical-equiv-cmon}
	The $\infty$-category $\ul\Mack_\gl(\Spc)$ of \emph{global Mackey functors} is presentable and equivariantly semiadditive, and the unique left adjoint $\ul\GammaS_\gl^\textup{spc}\to\ul\Mack_\gl(\Spc)$ sending $\mathbb P(1)$ to $\hom(1,-)\colon\Span_{\Forb}(\Fglo)\to\Spc$ is an equivalence. Moreover, this equivalence fits into a commutative diagram
	\begin{equation}\label{diag:compatible-fgt}
		\begin{tikzcd}
			\ul\GammaS_\gl^\textup{spc}\arrow[r,"\sim"]\arrow[d,"\mathbb U"'] & \ul\Mack_\gl(\Spc)\arrow[d,"\fgt"]\\
			\ul\S_\gl\arrow[r,"\sim"'] & \ul\Spc_\gl\rlap,
		\end{tikzcd}
	\end{equation}
	where $\fgt$ is the forgetful functor given by restriction along $\Fglo_{/-}^\op\hookrightarrow\Span_{\Forb}(\Fglo_{/-})$.
	\begin{proof}
		By \cite[Corollary~9.14]{CLL_Spans}, the pair $(\ul\Mack_\gl(\Spc),\hom(1,-))$ enjoys the same universal property as the pair $(\ul{\GammaS}_\gl^\textup{spc},\mathbb P(1))$ (see Theorem~\ref{thm:pres-univ-prop-Gamma} above), proving the first statement.

		For the compatibility with the forgetful functors we observe that chasing the adjunction from \cite[Corollary~7.24]{CLL_Spans} through the equivalences of Theorem~8.4 and Corollary~9.13 of \emph{op.\ cit.}\ shows that $\fgt$ has a left adjoint. We may therefore check commutativity of $(\ref{diag:compatible-fgt})$ after passing to left adjoints.
		
		By the Yoneda Lemma, the left adjoint of $\fgt$ has to send the terminal object (i.e.~the functor corepresented by $1\in\Glo^\op$) to the functor $\hom(1,-)$. We conclude that the diagram of left adjoints commutes when plugging in the terminal global space, so the claim follows from the universal property of $\ul\Spc_\gl$.
	\end{proof}
\end{theorem}

\begin{remark}
	By the universal property of commutative monoids \cite[Corollary~2.5]{groth-gepner-nikolaus}, postcomposition with the forgetful functor yields an equivalence $\ul\Mack_\gl(\CMon)\iso\ul\Mack_\gl(\Spc)$. At various points, it will be convenient to switch between these two perspectives on global Mackey functors, and we will freely do so.
\end{remark}

We close this subsection by recording the following lemma for later use:

\begin{lemma}\label{lemma:mackgl-products}
Let $\Cc$ be an $\infty$-category with finite products. Then $\ul\Mack_\gl(\Cc)$ has global products. Moreover, if $F\colon\Cc\to\Dd$ preserves finite products, then the postcomposition functor $\ul\Mack_\gl(F)\colon\ul\Mack_\gl(\Cc)\to\ul\Mack_\gl(\Dd)$ preserves global products.
\begin{proof}
	Arguing as in Example~\ref{ex:gl-glproducts}, this follows at once by observing that for every $f\colon G\to H$ in $\Fglo$ the adjunction
	\[
		(f^*)^\op\colon \Fglo_{/H}^\op\rightleftarrows \Fglo_{/G}^\op \noloc(f_!)^\op
	\]
	lifts to an adjunction $\Span_{\Forb}(\Fglo_{/H})\rightleftarrows\Span_{\Forb}(\Fglo_{/G})$ by \cite[Lemma~C.20]{BachmannHoyois2021Norms}.
\end{proof}	
\end{lemma}

\begin{remark}\label{rk:Mackey-alternative}
	Recall that we defined the functoriality of $\ul\Mack_\gl$ by restriction along the functors $\Span(f_!)$. Thanks to $2$-functoriality of $\Fun(-,-)$, we can then equivalently describe this functoriality as the left Kan extension along the left adjoints $\Span(f^*)$ from the previous lemma, cf. \cite[Remark 4.1.10]{CHLL_NRings}. Moreover, by straightening--unstraightening, the global category $\Fglo_{/-}\colon\Fglo^\op\to\Cat_\infty$ is equivalent to the Borel category $\Fglo^\flat$, and under this equivalence the global subcategory $\Fglo_{/-}\times_{\Fglo}\Forb$ corresponds to $\Forb^\flat$, see \cite[Remark~8.2]{LLP}. As an upshot, we can equivalently describe $\ul\Mack_\gl(\Cc)$ by $\Fun^\times(\Span_{\Forb^\flat}(\Fglo^\flat),\Cc)$ with functoriality via left Kan extension.
\end{remark}

\subsection{The stable story} Our reference model of stable $G$-global homotopy theory is based on \emph{symmetric spectra} (in simplicial sets) in the sense of \cite{hss}:

\begin{construction}
	We call a map $f\colon X\to Y$ of symmetric spectra with $G$-action a \emph{$G$-global weak equivalence} if $\phi^*f\colon\phi^*X\to\phi^*Y$ is an $H$-stable equivalence of $H$-equivariant symmetric spectra in the sense of \cite[Definition~2.35]{hausmann-equivariant} for every finite group $H$ and every homomorphism $\phi\colon H\to G$.
\end{construction}

Essentially by definition, restriction along any homomorphism $G\to G'$ sends $G'$-global weak equivalences to $G$-global weak equivalences. As before, we therefore get a global $\infty$-category $\ul\mySp_\gl$ by localizing the global Borel category associated to the 1-category of symmetric spectra. This again has a universal property:

\begin{definition}
	A global $\infty$-category $\Cc$ is called \emph{equivariantly stable} if it is equivariantly semiadditive and moreover \emph{fiberwise stable} in the sense that it factors through the non-full subcategory $\CAT_\infty^\text{ex}\subset\CAT_\infty$ of stable $\infty$-categories and exact functors.
\end{definition}

\begin{theorem}[See {\cite[Theorem~7.3.2]{CLL_Global}}]
	The global $\infty$-category $\ul\mySp_\gl$ is presentable and equivariantly stable. Moreover, for any other such $\Dd$, evaluating at the global sphere spectrum $\mathbb S\in\mySp_\gl$ determines an equivalence
	\[
		\Fun_\gl^\textup{L}(\ul\mySp_\gl,\Dd)\iso\Dd(1).\qednow
	\]
\end{theorem}

\begin{variant}
	Instead of localizing symmetric $G$-spectra at the \emph{$G$-global} weak equivalences, we can also localize them at the \emph{$G$-equivariant} weak equivalences (the \emph{$G$-stable weak equivalences} in Hausmann's terminology), yielding an $\infty$-category $\mySp_G$ of \emph{(genuine) $G$-equivariant spectra}. As the $G$-global weak equivalences are by definition finer than the $G$-equivariant ones, the identity functor derives to a localization $\fgt\colon\mySp_\text{$G$-gl}\to\mySp_G$.
	
	In the equivariant setting, the restriction functors are no longer homotopical, but they can be \emph{left} derived to produce a global $\infty$-category $\ul\mySp$, see \cite[Subsection~9.1]{CLL_Clefts} for a precise definition. The forgetful functors then have fully faithful left adjoints assembling into a global functor $\ul\mySp\hookrightarrow\ul\mySp_\gl$ sending the usual sphere to the global sphere, see Lemma~9.12 of \emph{op.\ cit.} Beware, however, that the forgetful functors on the other hand can \emph{not} be assembled into a global functor. 
\end{variant}

We can also give a Mackey functor description of global stable homotopy theory:

\begin{theorem}\label{thm:spectra-vs-spectral-Mackey}
	The global $\infty$-category $\ul\Mack_\gl(\Sp)$ is presentable and equivariantly stable. Moreover, the unique left adjoint $\ul\mySp_\gl\to\ul\Mack_\gl(\Sp)$ sending the global sphere $\mathbb S$ to the composite
	\[
		\Span_{\Forb}(\Fglo)\xrightarrow{\;\hom(1,-)\;}\CMon\xrightarrow{\;\ell\;}\Sp
	\]
	is an equivalence, where $\ell\colon\CMon\to\Sp$ denotes the delooping functor, i.e.~the unique left adjoint sending the free $E_\infty$-monoid to the sphere spectrum.
	\begin{proof}
		This follows immediately by combining the previous theorem with \cite[Corollary~9.17]{CLL_Spans}.
	\end{proof}
\end{theorem}

\begin{variant}
	Recall from Remark~\ref{rk:Mackey-alternative} that we can equivalently define $\ul\Mack_\gl(\Sp)$ via $\ul\Mack_\gl(\Sp)(G)\coloneqq\Fun^\times(\Span_{\Forb^\flat(G)}(\Fglo^\flat(G)),\Sp)$, with functoriality via left Kan extension. Analogously, we obtain a global category $\ul\Mack(\Sp)\coloneqq\Fun^\times(\Span(\mathbb F^\flat),\Sp)$ sending a finite group $G$ to the $\infty$-category $\Fun^\times(\Span(\mathbb F_G),\Sp)$ of \emph{$G$-equivariant spectral Mackey functors}.
\end{variant}

\begin{theorem}[See {\cite[Theorem~8.3]{LLP}}]\label{thm:spectra-gl-vs-equiv-vs-Mackey}
	\begin{enumerate}
		\item There exists a unique equivalence $\ul\mySp\iso\ul\Mack(\Sp)$ sending the sphere spectrum to the composite $\ell\circ\hom(1,-)\colon\Span(\mathbb F)\to\CMon\to\Sp$.
		\item Left Kan extension along $\mathbb F\hookrightarrow\Fglo$ defines a fully faithful functor $\ul\Mack(\Sp)\hookrightarrow\ul\Mack_\gl(\Sp)$, and the diagram
		\[
			\begin{tikzcd}
				\ul\mySp\arrow[r,hook]\arrow[d,"\sim"'] & \ul\mySp_\gl\arrow[d,"\sim"]\\
				\ul\Mack(\Sp)\arrow[r,hook] & \ul\Mack_\gl(\Sp)
			\end{tikzcd}
		\]
		commutes up to preferred equivalence, where the equivalence on the right is the one from Theorem~\ref{thm:spectra-vs-spectral-Mackey}.\qed
	\end{enumerate}
\end{theorem}

\subsection{Global algebraic $\bm{K}$-theory}
In this subsection, we will recall the definition of global algebraic $K$-theory from \cite{schwede-k} together with its $G$-global generalization \cite{g-global}. Before we come to this, however, let us recall the classical non-equivariant approach to the (group completion) $K$-theory of symmetric monoidal 1-categories:

\begin{construction}\label{constr:symmoncat-comp}
	In \cite[Definition~2.1]{shimada-shimakawa}, Shimada and Shimakawa (building on earlier work of Segal and May) constructed a functor $\Gamma\colon\cato{SymMonCat}^0\to\cato{$\bm\Gamma$-Cat}^\text{spc}_*$ from the 1-category of small symmetric monoidal 1-categories and \emph{strictly unital} strong symmetric functors to the 1-category of \emph{special $\varGamma$-categories}, i.e.~those reduced functors $\Gamma\to\textbf{Cat}$ into the 1-category of small 1-categories such that the Segal maps are {equivalences of categories}.

	We will only need to know the following facts about this construction:
	\begin{enumerate}
		\item The composite $\ev_{1^+}\circ\Gamma\colon\cato{SymMonCat}^0\to\cato{Cat}$ is isomorphic to the functor forgetting the symmetric monoidal structure; in particular, $\Gamma$ sends equivalences to levelwise equivalences. For ease of notation, we will pretend below that $\ev_{1^+}\circ\Gamma$ is actually \emph{equal} to the forgetful functor (this could for example be arranged by replacing $\Gamma$ by an isomorphic functor).
		\item Write $\cato{PermCat}$ for the 1-category of small \emph{permutative categories} (symmetric monoidal categories in which associativity and unitality hold strictly) and \emph{strict} symmetric monoidal functors. Then the homotopical functor $\Gamma$ descends to an equivalence of $\infty$-categories 
		\[
			\qquad\qquad\cato{PermCat}[\textup{equivalences}^{-1}]\iso\cato{$\bm\Gamma$-Cat}^\text{spc}_*[\textup{level equivalences}^{-1}],
		\]
		see \cite[Corollary~6.19]{sharma}. By \cite[Theorem~1.19]{ps-cat} and \cite[Proposition~6.7]{ps-global}, the inclusions $\cato{PermCat}\hookrightarrow\cato{SymMonCat}^0\hookrightarrow\cato{SymMonCat}$ each admit homotopy inverses; in particular, we see that $\Gamma$ uniquely extends to a functor
		\[
			\qquad\qquad\Gamma\colon\cato{SymMonCat}\to\cato{$\bm\Gamma$-Cat}^\text{spc}_*[\textup{level equivalences}^{-1}]
		\]
		that is a localization at the underlying equivalences of categories.

		Postcomposing with the nerve, this allows us to identify the localization $\cato{SymMonCat}[\text{equivalences}^{-1}]$ with the full subcategory $\SymMonCat_1\subset\SymMonCat_\infty=\CMon(\Cat_\infty)\subset\Fun(\Gamma,\Cat_\infty)$ spanned by those symmetric monoidal $\infty$-categories whose underlying category is a 1-category.
	\end{enumerate}

	Note that by the universal property of commutative monoids \cite[Corollary~2.5]{groth-gepner-nikolaus}, these two properties actually pin down the functor $\Gamma$ uniquely on the level of $\infty$-categories.
\end{construction}

One of the classical constructions of (group completion) algebraic $K$-theory is then given by the composite
\[
	\mathbb K\colon\cato{SymMonCat}^0\xrightarrow{\;\Gamma\;}\cato{$\bm\Gamma$-Cat}_*^\textup{spc}\xrightarrow{\;N\circ-\;}\cato{$\bm\Gamma$-SSet}_*^\textup{spc}\xrightarrow{\;\textup{deloop}\;}\cato{Spectra},
\]
where $\text{deloop}$ denotes any pointset model of the delooping functor $\ell\colon\CMon\to\Sp$, e.g.\ the one from \cite[Section~5]{bousfield-friedlander}. With our pointset models of $G$-global homotopy theory at hand, the construction of \emph{$G$-global algebraic $K$-theory} is then a rather direct adaptation of this:

\begin{construction}
	Fix a finite group $G$. We write $\Gamma_\gl$ for the composite
	\[
		\hskip-7.54pt\hfuzz=7.55pt
		\cato{$\bm G$-SymMonCat}^0\xrightarrow{\,\Gamma\,} \cato{$\bm\Gamma$-$\bm G$-Cat}_*\xrightarrow{\Fun(E\mathcal M,-)\,}\cato{$\bm\Gamma$-$\bm{E\mathcal M}$-$\bm G$-Cat}_*\xrightarrow{\,N\,}\cato{$\bm\Gamma$-$\bm{E\mathcal M}$-$\bm G$-SSet}_*
	\]
	where $E\mathcal M$ acts on itself from the right in the obvious way. By \cite[Example~2.2.52]{g-global}, this actually lands in the full subcategory of \emph{special} $G$-global $\Gamma$-spaces; as it is moreover strictly natural in $G$ by construction, varying $G$ yields a global functor $(\cato{SymMonCat}^0)^\flat\to \ul\GammaS_\gl^\text{spc}$, with a unique extension $\cato{SymMonCat}^\flat\to\ul\GammaS_\gl^\text{spc}$. We will denote both of these global functors by $\Gamma_\gl$ again.
\end{construction}

To complete the construction, we need the appropriate global version of the delooping functor:

\begin{construction}\label{constr:deloop-gl}
	We write $\cato{$\bm\Gamma$-$\bm G$-$\bm{\mathcal I}$-SSet}$ for the category of reduced functors $\Gamma\to\cato{$\bm G$-$\bm{\mathcal I}$-SSet}$. This comes with similar notions of \emph{$G$-global level weak equivalences} and \emph{specialness} \cite[Theorem~2.2.31 and Definition~2.2.50]{g-global}, both of which we will treat as a blackbox. By Corollary~2.2.53 of \emph{op.\ cit.}, there exists an (explicit) functor denoted $(-)[\omega^\bullet]\colon\cato{$\bm\Gamma$-$\bm{E\mathcal M}$-SSet}\to\cato{$\bm\Gamma$-$\bm{\mathcal I}$-SSet}$ inducing for any finite $G$ an equivalence
	\[
		\cato{$\bm\Gamma$-$\bm{E\mathcal M}$-$\bm G$-SSet}[\Ww_\text{$G$-gl}^{-1}]\iso\cato{$\bm\Gamma$-$\bm G$-$\bm{\mathcal I}$-SSet}[\Ww_\text{$G$-gl}^{-1}]
	\]
	and hence an equivalence $\ul\GammaS_\gl^\text{spc}\iso \ul\GammaS_{\gl,\Ii}^\text{spc}$ of the corresponding global $\infty$-categories.

	On the other hand, \cite[Construction~3.4.13]{g-global} provides a homotopical functor $\Ee^\otimes\colon\cato{$\bm\Gamma$-$\bm G$-$\bm{\mathcal I}$-SSet}\to\cato{$\bm G$-Spectra}$ to the 1-category of $G$-symmetric spectra, and for varying $G$ these assemble into a global left adjoint $\Ee^\otimes\colon\ul\GammaS_{\gl,\Ii}^\text{spc}\to\ul\mySp_\gl$. We write $\text{deloop}$ for the composite left adjoint\footnote{The corresponding pointset level functor is denoted $(-)\langle\mathbb S\rangle$ in \cite[Construction~4.1.6]{g-global}.}
	\[
		\ul\GammaS_\gl^\text{spc}\iso\ul\GammaS_{\gl,\Ii}^\text{spc}\xrightarrow{\;\Ee^\otimes\;}\ul\mySp_\gl.
	\] 
	Thanks to the universal property of the source, this functor can also simply be described as the unique left adjoint sending $\mathbb P(1)$ to the global sphere, see \cite[proof of Theorem~7.3.1]{CLL_Global}.
\end{construction}

We remark purely for motivational purposes that by \cite[Theorem~3.4.21]{g-global}, $\Ee^\otimes$ defines a Bousfield localization onto a suitable subcategory of \emph{connective $G$-global spectra}, as one would expect from a delooping functor.

\begin{definition}[see {\cite[Definition~4.1.31]{g-global}}]
	The \emph{($G$-)global algebraic $K$-theory functor} is the composite
	\begin{equation}\label{eq:Kggl}
		\mathbb K_\gl\colon \cato{SymMonCat}^\flat\xrightarrow{\;\Gamma_\gl\;}\ul\GammaS_\gl^\text{spc}\xrightarrow{\;\text{deloop}\;}\ul\mySp_\gl.
	\end{equation}
\end{definition}

\begin{remark}
	Specializing to $G=1$, the resulting functor $\cato{SymMonCat}\to\mySp_\gl$ agrees with the construction of \emph{global algebraic $K$-theory} from \cite{schwede-k}, see \cite[Remark~4.1.26 and Theorem~4.1.33]{g-global}. On the other hand, as we will explain in more detail in the final section of this paper, the functor $\cato{$\bm G$-SymMonCat}\to\mySp_\text{$G$-gl}$ for general $G$ provides a refinement of the equivariant algebraic $K$-theory constructions of \cite{guillou-may} and \cite{merling}. As promised in the introduction, this will allow us to prove both Theorems~\ref{introthm:equivariant} and~\ref{introthm:global} by studying the global functor $(\ref{eq:Kggl})$.
\end{remark}

\section{The non-group-completed comparison}\label{sec:heart}
\subsection{Swan $\bm K$-theory} In this section we will prove the key input for our main results by providing an alternative description of the composite 
\[
	(\cato{SymMonCat})^\flat\xrightarrow{\;\Gamma_\gl\;}\ul\GammaS_\gl^\text{spc}\iso\ul\Mack_\gl(\Spc)
\]
in terms of a monoidal (and parametrized) version of the Borel construction from Example~\ref{ex:Borel}, cf.\ \cite{puetzstueck-new}. As the first step, note that on the left hand side we still use the \emph{1-category} of small symmetric monoidal 1-categories, and we should hurry to replace this by a homotopy invariant notion:

\begin{remark}
	Recall from Construction~\ref{constr:symmoncat-comp} the localization $\gamma\colon\cato{PermCat}\to\cat{SymMonCat}_1$ (given by the composite $N\circ\Gamma$) and its extension to a localization $\gamma\colon\cato{SymMonCat}\to\cat{SymMonCat}_1$, obtained by precomposing with a homotopy inverse to the inclusion $\cato{PermCat}\hookrightarrow\cato{SymMonCat}$.

	By \cite[Theorem~3.1]{sharma}, the underlying equivalences of categories are part of a combinatorial simplicial model structure on $\cato{PermCat}$; \cite[Proposition~4.2.4.4]{HTT} thus shows that $\Fun(I,\gamma)\colon\Fun(I,\cato{PermCat})\to\Fun(I,\cat{SymMonCat}_1)$ is a localization for any small category $I$. On the other hand, any choice of homotopy inverse $\cato{SymMonCat}\to\cato{PermCat}$ induces a homotopy inverse to the inclusion $\Fun(I,\cato{PermCat})\hookrightarrow\Fun(I,\cato{SymMonCat})$, so we altogether conclude that $\gamma^\flat\colon\cato{SymMonCat}^\flat\to\cat{SymMonCat}_1^\flat$ is given at each $G\in\Fglo$ by a localization at the underlying equivalences of categories.
\end{remark}

Having made this translation, we are now in familiar $\infty$-categorical territory. The next step will be to construct a global functor $\cat{SymMonCat}_1^\flat\to\ul\Mack_\gl(\Cat_1)$ given informally by sending a symmetric monoidal $1$-category with (homotopy coherent) $G$-action $\Cc$ to the functor $\Span_{\Forb}(\Fglo_{/G})\to\Cat_1$ given at a morphism $\phi\colon H\to G$ of $\Glo$ by the homotopy fixed points $\Cc^{h\phi}\coloneqq(\phi^*\Cc)^{hH}$. To do so, we begin by relating the Borel construction to its more genuine counterpart from Example~\ref{ex:sub-gl}.

\begin{construction}
	\label{constr:lowercase-s}
	Write $i$ for the inclusion $\Fglo\hookrightarrow\Cat$. Equivalently, this is the corepresented functor $\hom_{\Fglo}(1,-)$; thus, the Yoneda lemma shows that there exists a unique natural transformation $s\colon i\to(\Fglo_{/-})^\op$ such that $s_1\colon 1=i(1)\to\Fglo$ is the functor picking out the terminal object $1\in\Fglo$.
\end{construction}

\begin{lemma}
	Let $G$ be any finite groupoid. Then $s_G\colon G\to(\Fglo_{/G})^\op$ is fully faithful, and its essential image consists precisely of the objects of the form $1\to G$.
	\begin{proof}
		The cocartesian unstraightening of $i=\hom(1,-)$ is given by the forgetful functor $\Fglo_{1/}\to\Fglo$. On the other hand, the cocartesian unstraightening of $\Fglo_{/-}$ is given by the target map $\ev_1\colon\Ar(\Fglo)\to\Fglo$, with cocartesian edges given by those maps inverted by the other evaluation functor $\ev_0$. We therefore obtain a natural transformation $i\to\Fglo_{/-}$ as the straightening of the inclusion $\Fglo_{1/}\hookrightarrow\Ar(\Fglo)$; this is clearly fully faithful with essential image as described above. Applying $(-)^\op$ to both sides and using the natural identification $G\simeq G^\op$ for every $G\in\Fglo$ then yields a natural transformation $i\to(\Fglo_{/-})^\op$ given in degree $1$ by picking out $1\in\Fglo$. As the latter property characterizes $s$ completely, this finishes the proof.
	\end{proof}
\end{lemma}

Restriction along $s$ then gives a global functor $s^*\colon\ul\Cc_\gl\to\Cc^\flat$ for any $\infty$-category $\Cc$. If $\Cc$ is complete, then each $\ul\Cc_\gl(G)\to\Cc^\flat(G)$ has a right adjoint $s_*$ given by right Kan extension along $s$; for example, for $G=1$ and $\Cc=\Cat_\infty$, this precisely recovers the functor $(-)^\flat\colon\Cat_\infty\to\Cat_{\gl,\infty}$. These right adjoints actually assemble into a global functor:

\begin{lemma}\label{lemma:s-lc-ra}
	Let $\Cc$ be a complete $\infty$-category. Then the global functor $s^*\colon\ul\Cc_\gl\to\Cc^\flat$ admits a right adjoint $s_*$.
	\begin{proof}
		Let $f\colon G\to H$ be any map in $\Fglo$; by the pointwise criterion from Lemma~\ref{lemma:adjunctions-criterion} we have to show that the Beck--Chevalley map $f^*s_*\to s_*f^*$ associated to the naturality square
		\begin{equation}\label{diag:BC-pb-spc}
			\begin{tikzcd}
				G\arrow[d,"f"'] \arrow[r,"s_G"] & \Fglo_{/G}^\op\arrow[d,"\Fglo_{/f}^\op"]\\
				H\arrow[r, "s_H"'] & \Fglo_{/H}^\op
			\end{tikzcd}
		\end{equation}
		is invertible. For this we observe that $\Fglo_{/G}\to\Fglo_{/H}$ is a right fibration; the claim will therefore follow from smooth basechange \cite[Theorem~6.4.13]{cisinski-book} once we show that $(\ref{diag:BC-pb-spc})$ is a pullback. As both horizontal functors are fully faithful by the previous lemma, this amounts to saying that an object $(X\to G)\in\Fglo_{/G}^\op$ is contained in the essential image of $s_G$ if and only if its image in $\Fglo_{/H}^\op$ is contained in the essential image of $s_H$. This is clear since, by another application of the lemma, both of these conditions are equivalent to demanding $X\simeq1$.
	\end{proof}
\end{lemma}

Our alternative description of ($G$-)global algebraic $K$-theory relies on a monoidal refinement $\SymMonCat_\infty^\flat\to\ul\Mack_\gl(\Cat_\infty)$ of this right adjoint:

\begin{construction}\label{constr:monoidal-borel}
	Composing the fully faithful natural transformation $s$ from Construction~\ref{constr:lowercase-s} with the inclusion $(\Fglo_{/-})^\op\hookrightarrow\Span_{\Forb}(\Fglo_{/-})$, we obtain a natural transformation $i\to \Span_{\Forb}(\Fglo_{/-})$. As $\Span_{\Forb}(\Fglo_{/G})$ is semiadditive for every $G\in\Fglo$, this then extends uniquely to a natural transformation
	\[
		S\colon \Span(\mathbb F)\times i(-)\to\Span_{\Forb}(\Fglo_{/-})
	\]
	of functors $\Fglo\to\Cat$, such that each each individual functor $\Span(\mathbb F)\times G\to\Span_{\Forb}(\Fglo_{/G})$ preserves direct sums in the first variable.
\end{construction}

Restriction along $S$ then defines a global functor
\[
\ul\Mack_\gl(\Cat_\infty)=\Fun^\times(\Span_{\Forb}(\Fglo_{/-}),\Cat_\infty)\to
\Fun(i(-),\Fun^\times(\Span(\mathbb F),\Cat_\infty)).
\]
Identifying $\cat{SymMonCat}_\infty\simeq\Fun^\times(\Span(\mathbb F),\Cat_\infty)$ in the usual way \cite[Proposition~C.1]{BachmannHoyois2021Norms}, we may view $S^*$ as a functor $\ul\Mack_\gl(\Cat_\infty)\to\cat{SymMonCat}^\flat_\infty$. By construction, this fits into a commutative diagram
\begin{equation}\label{diag:s-vs-S}
	\begin{tikzcd}
		\ul\Mack_\gl(\Cat_\infty)\arrow[r, "S^*"]\arrow[d,"\fgt"'] & \cat{SymMonCat}_\infty^\flat\arrow[d,"\fgt"]\\
		\ul\Cat_{\gl,\infty}\arrow[r,"s^*"'] & \cat{Cat}_\infty^\flat.
	\end{tikzcd}
\end{equation}
Note moreover that under the equivalence $\ul\Mack_\gl(\Cat_\infty)\simeq\ul\Mack_\gl(\SymMonCat_\infty)$, $S^*$ just corresponds to the global functor $\ul\Mack_\gl(\SymMonCat_\infty)\to\SymMonCat_\infty^\flat$ restricting along $i(-)\to\Span_{\Forb}(\Fglo_{/-})$: namely, this can be checked after postcomposing with the forgetful functor $\smash{\SymMonCat_\infty^\flat\to\Cat_\infty^\flat}$, where the claim reduces to commutativity of $(\ref{diag:s-vs-S})$. 

\begin{proposition}\label{prop:s*-vs-S*}
	The global functor $S^*\colon\ul\Mack_\gl(\Cat_\infty)\to\cat{SymMonCat}_\infty^\flat$ admits a global right adjoint $S_*$, and the Beck--Chevalley transformation $\fgt\circ S_*\to s_*\circ\fgt$ is an equivalence.
	\begin{proof}
		It is clear that each individual functor
		\[
			S^*\colon\Fun(\Span_{\Forb}(\Fglo_{/G}),\Cat_\infty)\to\Fun(\Span(\mathbb F)\times G,\Cat_\infty)
		\]
		has a right adjoint $S_*$, given by right Kan extension. By \cite[Proposition~3.10 and Example~3.6]{puetzstueck-new}, the Beck--Chevalley map
		\[
			\begin{tikzcd}
				\Fun(G,\Fun^\times(\Span(\mathbb F),\Cat_\infty))\arrow[d,"\fgt"']\arrow[r, "S_*"] &\Fun(\Span_{\Forb}(\Fglo_{/G}),\Cat_\infty)\arrow[d,"\fgt"]\arrow[dl,Rightarrow,shorten=15pt]\\
				\Fun(G,\Cat_\infty)\arrow[r, "s_*"'] &\Fun(\Fglo_{/G}^\op,\Cat_\infty)
			\end{tikzcd}
		\]
		is an equivalence for each $G$, while Lemma~3.9 of \emph{op.\ cit.}\ shows that the right adjoints restrict to the appropriate subcategories of product-preserving functors. It then only remains to show that the pointwise right adjoints assemble into a global right adjoint $S_*$, which in light of the pointwise criterion from Lemma~\ref{lemma:adjunctions-criterion} amounts to saying that the Beck--Chevalley map $f^*S_*\to S_*f^*$ is invertible for every $f\colon G\to H$ in $G$. As the forgetful functor $\fgt\colon\ul\Mack_\gl(\Cat)\to\ul\Cat_\gl$ is conservative, it will be enough to show this after postcomposition with $\fgt$. However, by the compatibility of Beck--Chevalley maps with pastings \cite[Lemma~C.3${}^\op$]{CLL_Adams}, the resulting map $\fgt f^*S_*\to \fgt S_*f^*$ fits into a commutative diagram of Beck--Chevalley maps
		\[
			\begin{tikzcd}
				\fgt f^*S_*\arrow[r]\arrow[d,"\sim"'] & \fgt S_* f^*\arrow[d]\\
				f^*\fgt S_*\arrow[d]& s_*\fgt f^*\arrow[d,"\sim"]\\
				f^* s_*\fgt\arrow[r] & s_* f^*\fgt\rlap.
			\end{tikzcd}
		\]
		All the Beck--Chevalley maps except possibly the top horizontal one are invertible by the above and Lemma~\ref{lemma:s-lc-ra}; the claim now follows by 2-out-of-3.
	\end{proof}
\end{proposition}

\begin{construction}\label{constr:swan}
	We write $\Swan$ for the composite
	\[
		\cat{SymMonCat}_1^\flat\hookrightarrow\cat{SymMonCat}_\infty^\flat\xrightarrow{\;S_*\;}\ul\Mack_\gl(\Cat_\infty)\xrightarrow{\;\ul\Mack_\gl(B)\;}\ul\Mack_\gl(\Spc),
	\]
	where $B\colon\Cat_\infty\to\Spc$ denotes the classifying space functor (i.e.\ the left adjoint of the inclusion).
\end{construction}

\subsection{The comparison} We are now ready to state the main result of this section:

\begin{theorem}\label{thm:swan-equivalence}
	The diagram
	\begin{equation}\label{diag:to-commute}
		\begin{tikzcd}
			\cato{SymMonCat}^\flat\arrow[d,"\gamma^\flat"']\arrow[r,"\Gamma_\gl"] & \ul\GammaS_\gl^\textup{spc}\arrow[d,"\sim"]\\
			\cat{SymMonCat}_1^\flat\arrow[r,"\Swan"'] & \ul\Mack_\gl(\Spc)
		\end{tikzcd}
	\end{equation}
	commutes up to unique homotopy, where the unlabelled equivalence is the one from Theorem~\ref{thm:canonical-equiv-cmon}.
\end{theorem}

The proof will occupy this whole subsection. Naturally, we would like to use the universal property of $\ul\GammaS_\gl^\text{spc}$ or equivalently of $\ul\Mack_\gl(\Spc)$: if we knew that all functors in sight were right adjoints, we could pass to left adjoints and then appeal to Theorem~\ref{thm:pres-univ-prop-Gamma} to reduce commutativity of the diagram to chasing through the global Mackey functor $\hom(1,-)\colon\Span_{\Forb}(\Fglo)\to\Spc$ through the left adjoints.

Unfortunately however, apart from the equivalence on the right, none of the functors populating $(\ref{diag:to-commute})$ are actually right adjoints. The crucial insight will be that we can fix this by passing to a suitable localization of the categories on the left:

\begin{definition}
	An equivariant functor $f\colon\Cc\to\Dd$ of small $1$-categories with (strict) $G$-actions is called a \emph{$G$-global weak equivalence} if the induced map
	\[
	N\big({\Fun^H(EH,\phi^*\Cc)}\big)\to 
	N\big({\Fun^H(EH,\phi^*\Dd)}\big)
	\]
	is a weak homotopy equivalence for every finite group $H$ and every homomorphism $\phi\colon H\to G$. We will call a map in $\cato{$\bm G$-SymMonCat}$ a \emph{$G$-global weak equivalence} if its underlying map in $\cato{$\bm G$-Cat}$ is, and we write $\Ww_\text{$G$-gl}$ for the collection of $G$-global weak equivalences (in either category).
\end{definition}
We then have the following ($G$-)global version of the main result of \cite{mandell-Gamma}:

\begin{theorem}[See {\cite[Theorem~4.3.7]{g-global}}]\label{thm:global-mandell}
	The global functor $\Gamma_\gl$ descends to an equivalence $\cato{SymMonCat}^\flat[\Ww_\gl^{-1}]\iso\ul\GammaS_\gl^\textup{spc}$ of global $\infty$-categories.\qed
\end{theorem}

We now want to show that the other path through the diagram $(\ref{diag:to-commute})$ descends to a right adjoint $\cato{SymMonCat}^\flat[\Ww_\gl^{-1}]\to\ul\Mack_\gl(\Spc)$. The case of the left hand vertical map turns out to be rather straight-forward:

\begin{remark}\label{rk:we-htpy-inv-desc}
	Observe that any underlying equivalence of categories is in particular a $G$-global weak equivalence \cite[Corollary~3.7]{merling}. In particular, the localization $\cato{SymMonCat}^\flat\to\cato{SymMonCat}^\flat[\Ww_\gl^{-1}]$ factors uniquely through a global functor $\cat{SymMonCat}_1^\flat\to\cato{SymMonCat}^\flat[\Ww_\gl^{-1}]$, and this is given in each degree $G\in\Glo$ by a localization at the images of the $G$-global weak equivalences or, equivalently, at those maps $f\colon\Cc\to\Dd$ in $\Fun(G,\cat{SymMonCat}_1)$ such that each $B(f^{h\phi})\colon B(\Cc^{h\phi})\to B(\Dd^{h\phi})$ is an equivalence in the $\infty$-category $\Spc$.

	In fact, there is no reason to restrict to $G\in\Glo\subset\Fglo$: if $G\in\Fglo$ is arbitrary, then we can define \emph{$G$-global weak equivalences} in $\Fun(G,\cato{SymMonCat})$ or $\Fun(G,\SymMonCat_1)$ by decomposing $G$ as a coproduct of $1$-object groupoids $G_1,\dots,G_r$, and transporting the weak equivalences along the resulting equivalences to $\prod_{i=1}^r\cato{$\bm{G_i}$-Cat}$ and $\prod_{i=1}^r\Fun(G_i,\Cat_1)$, respectively. In the former case, these are equivalently those maps $f$ such that $N{\big({\Fun^H(EH,\phi^*f)}\big)}$ is a weak homotopy equivalence for every finite group $H$ and any map $H\to G$ in $\Fglo$, and similarly the weak equivalences in the second case are precisely those maps such that $B(\phi^*(f)^{hH})$ is an equivalence for every such $\phi$. Since Dwyer--Kan localization commutes with finite products, taking the pointwise localization of the global $\infty$-categories $\cato{SymMonCat}^\flat\colon\Fglo^\op\to\Cat_\infty$ or $\SymMonCat_1^\flat\colon\Fglo^\op\to\Cat_\infty$ at these weak equivalences again results in product preserving functors $\Fglo^\op\to\Cat_\infty$. Thus, the above descriptions and results hold more generally for $G\in\Fglo$ and not just $G\in\Glo$.
\end{remark}

It then remains to show:

\begin{proposition}\label{prop:swan-ra}
	The global functor $\Swan$ descends to a right adjoint \[\SymMonCat_1^\flat[\Ww_\gl^{-1}]\to\ul\Mack_\gl(\Spc).\]
\end{proposition}

The proof of this will be significantly harder, and it naturally splits into three parts: we have to prove that $\Swan$ actually descends, we have to show the existence of pointwise adjoints, and we will have to verify the Beck--Chevalley condition from Lemma~\ref{lemma:adjunctions-criterion}. Recall that for the latter condition, we may instead show that $\Swan$ preserves global products, and we begin with some preparations for this result:

\begin{lemma}
	The global category $\cat{SymMonCat}_1^\flat[\Ww_\gl^{-1}]$ has global products, and the localization functor $\cat{SymMonCat}_1^\flat\to\cat{SymMonCat}_1^\flat[\Ww_\gl^{-1}]$ preserves global products.
	\begin{proof}
		As $\SymMonCat_1^\flat$ has global products (Example~\ref{ex:Borel-prod}), both claims will follow from showing that the functor $f_*\colon\Fun(G,\SymMonCat_1)\to\Fun(H,\SymMonCat_1)$ sends $G$-global weak equivalences to $H$-global weak equivalences for every $f\colon G\to H$ in $\Fglo$.

		For this let $\phi\colon K\to H$ be any further map in $\Fglo$, and consider the pullback
		\[
			\begin{tikzcd}
				P\arrow[r,"\psi"]\arrow[d,"g"']\arrow[dr,"\lrcorner"{very near start},phantom] & G\arrow[d,"f"]\\
				K\arrow[r,"\phi"'] & H\rlap.
			\end{tikzcd}
		\]
		As every map of groupoids is both a left and a right fibration, we conclude from smooth/proper basechange that \[\phi^*f_*\simeq g_*\psi^*\colon\Fun(G,\SymMonCat_1)\rightrightarrows\Fun(K,\SymMonCat_1),\] and hence $(-)^{h\phi} f_*=(-)^{hK}\phi^*f_*\simeq (-)^{hP}\circ\psi^*$ as right Kan extensions compose. Writing $P$ as a finite disjoint union of 1-object groupoids then completes the proof.
	\end{proof}
\end{lemma}

\begin{lemma}
	The global functor $\cat{SymMonCat}_1^\flat\to\ul\Mack_\gl(\Spc)$ induced by $\Swan$ preserves global products.
	\begin{proof}
 		Each of the individual functors from Construction~\ref{constr:swan} preserves global products: for the first two functors this is clear, while for the final functor this is an instance of Lemma~\ref{lemma:mackgl-products}.
	\end{proof}
\end{lemma}

\begin{corollary}\label{cor:swan-pw}
	Let $\phi\colon H\to G$ be any homomorphism of finite groups. Then the composite
	\[
		\Fun(G,\cat{SymMonCat}_1)\xrightarrow{\;\Swan\;}\ul\Mack_\gl(\Spc)(G)\xrightarrow{\;\ev_\phi\;}\Spc
	\]
	agrees with $B((-)^{h\phi})$.
	\begin{proof}
		We have
		\[
			\Swan(\Cc)(\phi)\simeq \Swan(\phi^*\Cc)(\id_H) \simeq\Swan(\phi^*(\Cc)^{hH})(\id_1)\stackrel{\text{def}}{=}\Swan(\Cc^{h\phi})(\id_1)
		\]
		naturally in $\Cc\in\Fun(G,\cat{SymMonCat}_1)$, where the first equivalence uses that $\Swan$ is a global functor together with the explicit description of the global functoriality of both sides, while the second equivalence follows similarly from the previous lemma.	We have therefore in particular reduced to the case $G=1$. In this case, the claim follows from the description of $(-)^\flat$ in Example~\ref{ex:Borel} together with Proposition~\ref{prop:s*-vs-S*}, but it is also part of \cite[Example~3.1]{puetzstueck-new}.
	\end{proof}
\end{corollary}

\begin{corollary}\label{cor:swan-descends-gl-prod-pres}
	The global functor $\Swan$ descends to a global-product-preserving functor $\cat{SymMonCat}_1^\flat[\Ww_\gl^{-1}]\to\ul\Mack_\gl(\Spc)$.
	\begin{proof}
		The previous corollary shows that $\Swan$ descends, while the two lemmas show that the resulting functor preserves global products.
	\end{proof}
\end{corollary}

It remains to show the existence of pointwise left adjoints. We want to appeal to the Adjoint Functor Theorem, so we need to get a handle on limits and filtered colimits in $\Fun(G,\cat{SymMonCat})[\Ww_\text{$G$-\gl}^{-1}]$ or equivalently $\cato{$\bm G$-SymMonCat}[\Ww_\text{$G$-\gl}^{-1}]$. For this we begin by observing that the upper composite in $(\ref{diag:to-commute})$ descends to an equivalence, which immediately shows the \emph{existence} of limits and colimits. In order to actually characterize them, we then have to understand this composite a bit better:

\begin{lemma}\label{lemma:top-composite-pw}
	Let $\phi\colon H\to G$ be any map of finite groups. Then the composite
	\begin{equation}\label{eq:limit-detecting}
		\cato{$\bm G$-SymMonCat}\xrightarrow{\;\Gamma_\gl\;}\GammaS_\gl^\textup{spc}\iso\Fun^\times(\Span_{\Forb}(\Fglo_{/G}),\Spc)\xrightarrow{\;\ev_\phi\;}\Spc
	\end{equation}
	agrees with the derived functor of
	\[
		\cato{$\bm G$-SymMonCat}\xrightarrow{(-)^{h\phi}}\cato{Cat}\xrightarrow{\;N\;}\cato{SSet}.
	\]
	\begin{proof}
		In the diagram
		\[
			\begin{tikzcd}
				\cato{$\bm G$-SymMonCat}\arrow[d,"\fgt"']\arrow[r,"\Gamma_\text{$G$-gl}"] &[1.5em] \GammaS_\text{$G$-gl}^\text{spc}\arrow[d,"\fgt"{description}]\arrow[r,"\sim"] & \ul\Mack_\gl(\Spc)(G)\arrow[d,"\fgt"]\\
				\cato{$\bm G$-Cat}\arrow[r, "{N\circ\Fun(E\mathcal M,-)}"'] & \S_\text{$G$-gl}\arrow[r,"\sim","\Phi"'] & \Spc_\text{$G$-\gl}
			\end{tikzcd}
		\]
		the left hand square commutes by definition while the right hand square commutes by Theorem~\ref{thm:canonical-equiv-cmon}. The claim now follows from the pointwise description of the equivalence $\Phi$ (Remark~\ref{rk:unstable-pw-desc}) together with the natural weak equivalence $\Fun^H(E\mathcal M,-)\simeq\Fun^H(EH,-)$ of functors $\cato{$\bm H$-Cat}\to\cato{Cat}$ from \cite[Remark~4.1.28]{g-global}.
	\end{proof}
\end{lemma}

\begin{lemma}\label{lemma:charact-limits}
	Let $K$ be any small $\infty$-category and let $G\in\Glo$. A diagram $K^\triangleleft\to\Fun(G,\cat{SymMonCat}_1)[\Ww_\textup{$G$-gl}^{-1}]$ is a limit diagram if and only if the composite with 
	\begin{equation}\label{eq:B-htpy-fixed-points}
		\Fun(G,\cat{SymMonCat}_1)\xrightarrow{\;(-)^{h\phi}\;}\cat{SymMonCat}_1\xrightarrow{\;B\;}\Spc
	\end{equation}
	is a limit diagram for every $\phi\colon H\to G$ in $\Glo$. The analogous characterization of \emph{sifted} colimits also holds.
	\begin{proof}
		We will focus on the first statement.
		
		We may equivalently replace $\Fun(G,\cat{SymMonCat}_1)$ by $\cato{$\bm G$-SymMonCat}$; as in Remark~\ref{rk:we-htpy-inv-desc}, the composite $(\ref{eq:B-htpy-fixed-points})$ translates to the derived functor of
		\begin{equation}\label{eq:underived-desc}
			\cato{$\bm G$-SymMonCat}\xrightarrow{\;\Fun^H(EH,\phi^*(-))\;}\cato{Cat}\xrightarrow{\;N\;}\cato{SSet}
		\end{equation}
		in this model. 
		
		As equivalences create limits, and since limits in $\Fun^\times(\Span_\Forb(\Fglo_{/G}),\Spc)$ are pointwise, $K^\triangleleft\to\cato{$\bm G$-SymMonCat}[\Ww_\textup{$G$-gl}^{-1}]$ is a limiting cone if and only if it becomes a limiting cone after postcomposing with $(\ref{eq:limit-detecting})$ for every $\phi\colon H\to G$. By the previous lemma, this precisely agrees with the derived functor of $(\ref{eq:underived-desc})$, finishing the proof of the claim for limits.

		The proof of the second claim is completely analogous, using that \emph{sifted} colimits in $\Fun^\times(\Span_\Forb(\Fglo_{/G}),\Spc)$ are again pointwise.
	\end{proof}
\end{lemma}

Putting all these pieces together we get:

\begin{proof}[Proof of Proposition~\ref{prop:swan-ra}]
	By Corollary~\ref{cor:swan-descends-gl-prod-pres}, $\Swan$ descends to a global product preserving functor. By Lemma~\ref{lemma:adjunctions-criterion}, it will then suffice to show that for each $G\in\Glo$ the resulting functor $\Fun(G,\cat{SymMonCat}_1)[\Ww_\gl^{-1}]\to\ul\Mack_\gl(\Spc)(G)$ is a right adjoint. As both sides are presentable (see Theorem~\ref{thm:global-mandell} for the source), it will suffice by the Adjoint Functor Theorem that this functor preserves limits as well as filtered colimits. As limits and filtered colimits in $\ul\Mack_\gl(\Spc)(G)$ are pointwise, this follows at once by combining the characterization of these limits and colimits from Lemma~\ref{lemma:charact-limits} with the pointwise description of $\Swan$ from Corollary~\ref{cor:swan-pw}.
\end{proof}

\begin{proof}[Proof of Theorem~\ref{thm:swan-equivalence}]
	By the universal property of localization, we may equivalently show that 
	\begin{equation}\label{diag:to-commute-after-loc}
		\begin{tikzcd}
			\cato{SymMonCat}^\flat[\Ww_\gl^{-1}]\arrow[d,"\sim","\gamma^\flat"']\arrow[r,"\Gamma_\gl"] & \ul\GammaS_\gl^\text{spc}\arrow[d,"\sim"]\\
			\cat{SymMonCat}^\flat_1[\Ww_\gl^{-1}]\arrow[r,"\Swan"'] & \ul\Mack_\gl(\Spc)
		\end{tikzcd}
	\end{equation}
	commutes up to unique homotopy. By Theorem~\ref{thm:global-mandell} and Proposition~\ref{prop:swan-ra} all functors in sight are right adjoints, so we may equivalently show this after passing to left adjoints.

	Let us first show that the diagram of left adjoints commutes up to \emph{some} homotopy. By the universal property of $\ul\Mack_\gl(\Spc)\simeq\ul\GammaS_\gl^\text{spc}$, it will be enough to chase through the element $\hom(1,-)\in\Mack_\gl(\Spc)$. By the Yoneda Lemma, this corepresents the functor $\Mack_\gl(\Spc)\to\Spc$ given by evaluation at 1. Passing to right adjoints again, we have therefore reduced to showing that the diagram
	\[
		\begin{tikzcd}
			\cato{SymMonCat}\arrow[d,"\gamma"']\arrow[r,"\Gamma_\gl"] & \GammaS_\gl^\text{spc}\arrow[d,"\sim"]\\
			\cat{SymMonCat}_1\arrow[r,"\Swan"'] & \Mack_\gl(\Spc)
		\end{tikzcd}
	\]
	commutes after postcomposition with $\ev_1\colon\Mack_\gl(\Spc)\to\Spc$. This, however, follows at once by the pointwise descriptions from Lemma~\ref{lemma:top-composite-pw} and Corollary~\ref{cor:swan-pw}.

	It remains to show uniqueness of the homotopy filling $(\ref{diag:to-commute-after-loc})$. As we already know that it commutes for \emph{some} homotopy, 2-out-of-3 shows that all maps comprising it are equivalences. Thus, uniqueness of the homotopy amounts to saying that the identity functor of $\ul\Mack_\gl(\Spc)$ has a contractible space of automorphisms. Appealing to the universal property from Theorem~\ref{thm:pres-univ-prop-Gamma}, this amounts to saying that $\hom(1,-)\colon\Span_{\Forb}(\Fglo)\to\Spc$ has contractible space of automorphisms. By the Yoneda lemma, the space of \emph{endomorphisms} is given by $\hom_{\Span_{\Forb}(\Fglo)}(1,1)$, so that the space of \emph{automorphisms} is given by $\hom_{\core\Span_{\Forb}(\Fglo)}(1,1)\simeq\hom_{\core\Fglo^\op}(1,1)$, which is indeed contractible as $1$ is terminal in $\Fglo$.
\end{proof}

Let us record the following observation from the above proof separately:

\begin{corollary}
	The global functor $\Swan$ descends to an equivalence
	\[
		\cat{SymMonCat}_1^\flat[\Ww_\gl^{-1}]\iso\ul\Mack_\gl(\Spc).\qednow
	\]
\end{corollary}

\section{Proof of main results}\label{sec:main-results}
In this final section, we will prove our comparison of global and equivariant algebraic $K$-theory constructions. With Theorem~\ref{thm:swan-equivalence} at hand, this will be rather straightforward.

\subsection{The global comparison} We begin by proving the following generalization of Theorem~\ref{introthm:global} from the introduction:

\begin{theorem}\label{thm:main-comp}
	The diagram
	\begin{equation}\label{diag:main-comp}
		\begin{tikzcd}
			\cato{SymMonCat}^\flat\arrow[d,"\gamma^\flat"']\arrow[rr, "\mathbb K_\gl"] &[-1em]&[1.75em] \ul\mySp_\gl\arrow[d,"\sim"]\\
			\cat{SymMonCat}_1^\flat\arrow[r, "S_*"'] &\ul\Mack_\gl(\SymMonCat_1) \arrow[r,"\ul\Mack_\gl(\mathbb K)\,"'] &\ul\Mack_\gl(\Sp)
		\end{tikzcd}
	\end{equation}
	commutes up to preferred equivalence; here the unlabelled equivalence on the right is the one from Theorem~\ref{thm:spectra-vs-spectral-Mackey}.
	\begin{proof}
		Write $\ell\colon\CMon\to\Sp$ for the delooping functor again. Plugging in the definitions, we can rewrite the composite of the bottom row as
		\[
			\cat{SymMonCat}_1^\flat\xrightarrow{\;\Swan\;}\ul\Mack_\gl(\Spc)\iso\ul\Mack_\gl(\CMon)\xrightarrow{\;\ul\Mack_\gl(\ell)\;}\ul\Mack_\gl(\Sp).
		\]		
		On the other hand, $\mathbb K_\gl$ was defined as the composite
		\[
			\cato{SymMonCat}^\flat\xrightarrow{\;\Gamma_\gl\;}
			\ul\GammaS_\gl^\textup{spc}\xrightarrow{\;\textup{deloop}\;}\ul\mySp_\gl.
		\]
		In light of Theorem~\ref{thm:swan-equivalence}, it will therefore suffice to show that the diagram
		\[
			\begin{tikzcd}
				\ul\GammaS_\gl^\text{spc}\arrow[r,"\textup{deloop}"]\arrow[d,"\sim"'] &[2em] \ul\mySp_\gl\arrow[d,"\sim"]\\
				\ul\Mack_\gl(\Spc)\arrow[r,"\ul\Mack_\gl(\ell)"'] & \ul\Mack_\gl(\Sp)
			\end{tikzcd}
		\]
		commutes up to preferred equivalence. We saw as part of Construction~\ref{constr:deloop-gl} that the upper horizontal map is a left adjoint; moreover, also the bottom horizontal map is a left adjoint, with right adjoint given by postcomposing with the right adjoint of $\ell$. By Theorem~\ref{thm:pres-univ-prop-Gamma}, it will thus suffice to chase through the object $\mathbb P(1)\in\GammaS_\gl^\text{spc}$.

		By the definition of the left hand vertical map, the bottom left composite sends $\mathbb P(1)$ to the delooping of $\hom(1,-)\colon\Span_{\Forb}(\Fglo)\to\CMon$, which is in turn the image of the global sphere under the right hand vertical equivalence, by definition of the latter. Finally, we recalled in Construction~\ref{constr:deloop-gl} that the delooping of $\mathbb P(1)$ is equivalent (in a preferred way) to the global sphere.
	\end{proof}
\end{theorem}

\begin{remark}
	Unlike in Theorem~\ref{thm:swan-equivalence}, the space of equivalences filling $(\ref{diag:main-comp})$ is no longer contractible; instead, combining Theorem~\ref{thm:global-mandell} with the universal property of $\ul\GammaS_\gl^\text{spc}$ shows that it is a torsor over the $E_\infty$-group of automorphisms of the (global, or equivalently non-equivariant) sphere spectrum, i.e.~over $\text{gl}_1(\mathbb S)$.
\end{remark}

\subsection{The equivariant comparison} Fix a finite group $G$. Our final goal is to prove Theorem~\ref{introthm:equivariant} from the introduction, describing equivariant algebraic $K$-theory \`a la \cite{guillou-may,merling} in terms of the \emph{$G$-equivariant monoidal Borel construction}. The former is a specific pointset level functor assigning to a symmetric monoidal 1-category $\Cc$ with $G$-action an orthogonal spectrum $\mathbb K_G(\Cc)$ with $G$-action, or equivalently a symmetric spectrum with $G$-action, viewed through the eyes of the \emph{$G$-equivariant weak equivalences}. The details of this construction will actually not be relevant for our purposes because of the following comparison result:

\begin{proposition}[See {\cite[Theorem~4.1.40]{g-global}}]\label{prop:equiv-vs-gl}
	The diagram
	\[
		\begin{tikzcd}
			\cato{$\bm G$-SymMonCat}\arrow[dr, bend right=15pt,"\mathbb K_G"']\arrow[r,"\mathbb K_\textup{$G$-gl}\,"] &[.5em] \mySp_\textup{$G$-gl}\arrow[d,"\fgt"]\\
			& \mySp_G
		\end{tikzcd}
	\]
	commutes up to preferred equivalence, where the vertical map on the right again denotes the derived functor of the identity.\qed
\end{proposition}

On the other hand, Construction~\ref{constr:monoidal-borel} has the following $G$-equivariant analogue, see e.g.\ \cite[Construction~8.1]{spectral-mackey-two}, \cite[Proposition~3.3.3]{hilman2022parametrised}, or \cite[Example~3.6(i)]{puetzstueck-new}:

\begin{construction}
	Consider $G$ as a 1-object groupoid again. We have a fully faithful inclusion $G^\op\hookrightarrow\mathbb F_G$ sending the unique object to the transitive free $G$-set $G$, and a group element $g$ to the map $G\to G, x\mapsto xg$. The composite $BG\to \mathbb F_G^\op\hookrightarrow\Span(\mathbb F_G)$ then extends uniquely to a functor $T\colon\Span(\mathbb F)\times BG\to\Span(\mathbb F_G)$ preserving direct sums in the first variable, and right Kan extension along this provides us with a functor
	\[\hskip-16.378pt\hfuzz=16.378pt
		T_*\colon\Fun(G,\cat{SymMonCat}_\infty)\to\Fun^\times(\Span(\mathbb F_G),\Cat_\infty)\simeq\Fun^\times(\Span(\mathbb F_G),\SymMonCat_\infty)\rlap.
	\]
\end{construction}

We then have the following precise version of Theorem~\ref{introthm:equivariant} from the introduction:

\begin{theorem}\label{thm:equiv-comp}
	The diagram
	\[
		\begin{tikzcd}[cramped]
			\cato{$\bm G$-SymMonCat}\arrow[r,"\gamma"]\arrow[d, "{\mathbb{K}}_G"'] &[-1em] \Fun(G,\cat{SymMonCat}_1)\arrow[r,"T_*\,"] &[-.5em] \Fun^\times(\Span(\mathbb F_G),\SymMonCat_\infty)\arrow[d,"\mathbb K\circ-"]\\
			\mySp_G\arrow[rr,"\sim"'] && \Fun^\times(\Span(\mathbb F_G),\Sp)
		\end{tikzcd}
	\]
	commutes up to preferred equivalence. Here the bottom equivalence is the one from Theorem~\ref{thm:spectra-gl-vs-equiv-vs-Mackey}.
	\begin{proof}
		By Proposition~\ref{prop:equiv-vs-gl}, the bottom composite is equivalent to
		\[
			\cato{$\bm G$-SymMonCat}\xrightarrow{\;\mathbb K_\text{$G$-gl}\;} \mySp_\text{$G$-gl}\xrightarrow{\;\fgt\;}\mySp_\text{$G$}\iso\Fun^\times\hskip-1pt(\Span(\mathbb F_G),\Sp)=\Mack_G(\Sp).
		\]
		Passing to right adjoints in Theorem~\ref{thm:spectra-gl-vs-equiv-vs-Mackey} shows that this is in turn equivalent to
		\[
			\cato{$\bm G$-SymMonCat}\xrightarrow{\;\mathbb K_\text{$G$-gl}\;}\mySp_\text{$G$-gl}\iso\Mack_\text{$G$-gl}(\Sp)\xrightarrow{\;\fgt\;}\Mack_\text{$G$}(\Sp),
		\]
		where the unlabelled equivalence $\mySp_{\text{$G$-gl}}\iso\Mack_\text{$G$-gl}(\Sp)$ is the one we considered throughout. We are therefore reduced to showing that the diagram
		\[
			\begin{tikzcd}[cramped]
				\cato{$\bm G$-SymMonCat}\arrow[dd,"\mathbb K_\text{$G$-gl}"']\arrow[r,"\gamma"] & \Fun(G,\SymMonCat_1)\arrow[d,"S_*"{description}]\arrow[dr, bend left=15pt,"T_*"]\\
				& \Mack_\text{$G$-gl}(\SymMonCat_\infty)\arrow[d,"\mathbb K\circ-"{description}]\arrow[r,"\fgt"{description}] & \Mack_G(\SymMonCat_\infty)\arrow[d,"\mathbb K\circ-"]\\
				\mySp_\text{$G$-gl} \arrow[r,"\sim"'] & \Mack_\text{$G$-\gl}(\Sp)\arrow[r,"\fgt"'] & \Mack_G(\Sp)
			\end{tikzcd}
		\]
		commutes up to preferred equivalence. However, for the left hand rectangle this is the content of Theorem~\ref{thm:main-comp}, the bottom right rectangle commutes simply by functoriality, and commutativity of the top right triangle is \cite[Proposition 3.24]{puetzstueck-new}.
	\end{proof}
\end{theorem}

\bibliography{reference}
\end{document}